\documentclass[a4paper]{article}
\usepackage[top=3cm, left=2cm, bottom=3cm, right=2cm]{geometry}
\usepackage[ngerman,english]{babel}
\usepackage{graphicx}
\usepackage{amsmath, amsthm, amssymb} 
\usepackage{overpic}
\usepackage{dsfont}
\usepackage[T1]{fontenc}
\usepackage{color}
\usepackage{natbib}
\usepackage[T1]{fontenc}
\usepackage{setspace}
\usepackage{upgreek}
\usepackage{bm}

\allowdisplaybreaks

\DeclareMathOperator\sgn{sgn}

\DeclareMathOperator\arcoth{arcoth}

\DeclareMathOperator\var{Var}

\newtheorem{theorem}{Theorem}[section]

\newtheorem{prop}[theorem]{Proposition}

\newtheorem{cor}[theorem]{Corollary}

\newtheorem{lem}[theorem]{Lemma}

\newtheorem{assumption}[theorem]{Assumption}

\newtheorem*{theorem*}{Theorem}
\newtheorem*{proposition*}{Proposition}
\newtheorem*{lemma*}{Lemma}
\newtheorem*{corollary*}{Corollary}
\newtheorem*{remark*}{Remark}

\newtheorem*{observation*}{Observation}
\newtheorem*{example*}{Example}
\newtheorem*{assumption*}{Assumption}

\theoremstyle{definition}
\newtheorem{definition}[theorem]{Definition}
\newtheorem*{definition*}{Definition}

\title{\begin{Huge}An explicit solution of a non-linear quadratic constrained stochastic control problem with an application to optimal liquidation in dark pools with adverse selection\end{Huge}\footnote{I wish to thank Peter Bank, Ulrich Horst, Werner Kratz and Torsten Sch\"oneborn for useful discussions and comments. I am also grateful to seminar participants at Humboldt University Berlin and at German Probability and Statistics Days Mainz.
This research was supported by Deutsche Bank through the Quantitative Products Laboratory.} \\
}

\author{Peter Kratz\footnote{Laboratoire d'Analyse, Topologie, Probabilit\'es, Aix-Marseille Universit\'e, 39, rue F. Joliot-Curie,
13453 Marseille Cedex 13, France. Email: kratz@mathematik.hu-berlin.de}}

\begin{document}

\maketitle

\begin{abstract}
\noindent
We study a constrained stochastic control problem with jumps; the jump times of the controlled process are given by a Poisson process. The cost functional comprises quadratic components for an absolutely continuous control and the controlled process and an absolute value component for the control of the jump size of the process. We characterize the value function by a ``polynomial'' of degree two whose coefficients depend on the state of the system; these coefficients are given by a coupled system of ODEs. The problem hence reduces from solving the Hamilton Jacobi Bellman (HJB) equation (i.e., a PDE) to solving an ODE whose solution is available in closed form. The state space is separated by a time dependent boundary into a continuation region where the optimal jump size of the controlled process is positive and a stopping region where it is zero. We apply the optimization problem to a problem faced by investors in the financial market who have to liquidate a position in a risky asset and have access to a dark pool with adverse selection. 
\end{abstract}

\section{Introduction}

For most non-linear-quadratic control problems, closed form solutions are rather difficult to obtain as the form of the value function cannot be guessed easily. In this paper, we propose an interpolation method for a non-linear-quadratic stochastic control problem with terminal constraint whose cost functional comprises both quadratic components and an absolute value component; we derive the solution of the problem in closed form and show that the value function is a classical solution of the corresponding Hamilton Jacobi Bellman (HJB) equation.

Given a control $u=(\xi,\eta)$, the controlled process $X^u$ is given by
\[
dX^u(t)= x- \int_0^ t \xi(s) ds - \int_0^t \eta(s) d\pi(s)
\]
and the terminal constraint
\begin{equation} \label{EqConstraintIntro}
\lim_{t\rightarrow T-} X^u(t)= 0;
\end{equation}
here, $x\in \mathds{R}$ is the initial state of the system at time $t \in [0,T]$ and $\pi$ is a Poisson process with intensity $\theta > 0$. The costs of $u$ are given by
\[
\mathbb{E} \Big[ \int_0^T \lambda \xi(t)^2  +  \gamma |\eta(t)| +  \alpha X^u(t)^2 ds \Big]
\]
for constants $\lambda,\gamma>0$ and $\alpha \geq 0$.

The solution of the control problem is driven by two trade-offs. Firstly, the trade-off between the quadratic costs of the control component $\xi$ and the absolute value costs of the control component $\eta$: in order to accomplish the Constraint~\eqref{EqConstraintIntro}, applying $\eta$ is more attractive (however uncertain) for large states while applying $\xi$ is more attractive for small states. Secondly, the trade-off between the cost component $\alpha (X^u)^2$ and the other two components: $\alpha (X^u)^2$ speeds up the controls while $\lambda \xi^2$ and  $\gamma |\eta|$ slow them down.

The complexity of the control problem stems from the combination of the Terminal Constraint~\eqref{EqConstraintIntro} and the non-linear quadratic form of the cost functional. It turns out that we require qualitative information of the speed of convergence of the controlled process $X^u$ as $t\rightarrow T-$ in order to conduct a verification because. We overcome this problem by constructing the solution of the control problem in closed form.

Because of the absolute value costs $\gamma |\eta|$, we cannot expect the value function $v$ of the problem to be a quadratic polynomial. However, it should in some sense be similar to a polynomial. By writing the value function as a ``polynomial'' whose coefficients depend on the state of the system,
\begin{equation} \label{EqIntroQuasi}
v(T,x)=C_1(T,x) x^2 + C_2(T,x) x + C_3(T,x),
\end{equation}
we are able to reduce the complexity of the problem: instead of characterizing $v$ as the solution of the Hamilton Jacobi Bellman (HJB) equation corresponding to the problem (i.e., a PDE), we are able to characterize it via~\eqref{EqIntroQuasi} and a system of ODEs for the coefficients $C_1$, $C_2$ and $C_3$; this system can be solved explicitly. 

It turns out that the value function is a true polynomial for small states $|x| \leq \beta(T)$ and for large states $|x| \geq \tilde{\beta}(T)$ (for functions $0 < \beta \leq \tilde{\beta}$). For intermediate states ($\beta(T) < |x| < \tilde{\beta}(T)$) the value function is given by an ``interpolation'' of these polynomials; we derive an implicit representation of the interpolation function as the inverse of an explicitly known injective function. The interpolation method we use appears promising for the solution of related control problems with both quadratic and absolute value cost components.\footnote{The application of the method to related control problems requires significant adjustments in order to account for problem specific properties, cf.~Section~\ref{SecAdvHeuristic}.} In the present case, the fact that the candidate of the value function and the optimal control is given in closed form is essential for the verification theorem as it enables us to resolve problems arising due to the singularity of the value function at time $T$ as mentioned above.

The optimal control of the jump size of the controlled process, $\eta^*$, divides the state space into two regions: a ``continuation region'' ($|x| > \beta(t)$) where $|\eta^*|>0$ and a ``stopping region'' ($|x| \leq \beta(T)$) where $\eta^*=0$. The structure of the optimal control hence resembles the solution of free boundary problems arising in optimal stopping (see, e.g., \cite{Pham2009}) and in singular control problems connected to optimal stopping (see, e.g., \cite{Benth2004}). The control problem studied in this paper is obviously not a singular control problem. Nevertheless, the similarity of the structure is not surprising due to the ``linearity'' of the costs for the jump size control $\eta$. Similarly, we prove a ``smooth-fit'' principle for the value function on the boundary $\beta$; it is essential in the heuristic derivation of the solution of the optimization problem. 

The paper contributes to the literature on stochastic control with jumps; standard references are the book by~\cite{OksendalSulem} and~\cite{Hanson2007}. Several papers study stochastic control problems with similar terminal constraints. \cite{Schied2010} study multi-dimensional CARA utility maximization without jumps. If $\gamma=0$, the constrained control problem is linear-quadratic; a multi-dimensional version of the problem was solved by~\cite{Kratz2012}. Finally, \cite{Naujokat2010} and~\cite{Naujokat2011} solve similar control problems by using the stochastic maximum principle and obtain presentations of the optimal controls via FBSDEs. I am not aware of any other work where the value function is represented by a quasi-polynomial and thus a closed form solution for a non-linear-quadratic stochastic control problem is obtained.
\\

We apply the solution of the optimization problem to a model of the financial market. In the last years, equity trading has been transformed by the advent of so called dark pools. These alternative trading venues differ significantly from classical exchanges and have gained a significant market share, especially in the US. Dark pools vary in a number of properties such as crossing procedure, ownership and accessibility (see~\cite{Mittal2008} and \cite{Degryse2009a} for further details and a typology of dark pools). However, they generally share the following two stylized facts. First, the liquidity available in dark pools is not quoted, hence making trade execution uncertain and unpredictable. Second, dark pools do not determine prices. Instead, they monitor the prices determined by the classical exchanges and settle trades in the dark pool only if possible at these prices. Thus, trades in the dark pool have no or less price impact.\footnote{For empirical evidence of lower transaction costs or price impact of dark pools compared to classical exchanges see, e.g., \cite{Conrad2003} and~\cite{Fong2004}.} 

We consider an investor who has access both to a classical exchange (also called ``primary venue'' or ``primary exchange'') and to a dark pool. We study a model for optimal liquidation of a large single-asset position within a finite time horizon $[0,T]$ reflecting the trade-off between execution uncertainty of dark pool orders and price impact costs of trading at the primary venue. Additionally, we assume that the asset price at the exchange and the execution of orders are connected by the following phenomenon: Liquidity seeking traders find that their trades in the dark pool are usually executed just before a favorable price move, i.e., exactly when they do not want them to be executed since they miss out on the price improvement. In advance of adverse price movements, they observe that they rarely find liquidity in the dark pool. We call such a phenomenon \emph{adverse selection}. Adverse selection in dark pools is an important issue for practitioners. There are several different mechanism through which adverse selection can be created.\footnote{A detailed discussion from a practical point of view can be found in~\cite{Mittal2008}, who states that information leakage, e.g., due to dark pool pinging (i.e., the attempt to obtain information about liquidity by placing small orders in the dark pool), is a possible reason for adverse selection.} While (absolute continuous) trading at the exchange yields quadratic price impact costs, adverse selection yields absolute value costs for the dark pool orders (whose execution is uncertain, modeled by a Poisson process); furthermore, we add a risk component which is quadratic in the asset position. We hence generalize the single-asset result of~\cite{Kratz2012} who neglect adverse selection in order to render the optimization problem linear-quadratic; the results we obtain are qualitatively different. Optimal liquidation in dark pools with adverse selection in discrete time has been treated in~\cite{Kratz2010}.
\\

The remainder of the article is structured as follows. In Section~\ref{ChaptPortfolioSecModel}, we describe the optimization problem rigorously. In Section~\ref{SecAdvHeuristic}, we derive candidates for the value function and the optimal control in closed form by extensive heuristic considerations. We analyze the candidate value function in Section~\ref{SecPropValAdv}: it is continuously differentiable and a classical solution of the HJB equation with singular terminal condition corresponding to the Terminal Constraint~\eqref{EqConstraintIntro} of the optimization problem. In Section~\ref{SecVerificationAdv}, we execute a verification argument. Finally, we apply the solution of the optimization problem to the above mentioned optimal liquidation problem in dark pools in Section~\ref{SecAdvProperties}. Furthermore, we discuss the properties of the optimal control with regard to the application.

\section{Optimization problem} \label{ChaptPortfolioSecModel}

In the following, we specify the optimization problem. 
For a fixed finite time interval $[0,T]$, we consider the stochastic basis $(\Omega, \mathcal{F}, \mathbb{P}, \mathbb{F}=(\mathcal{F}_t)_{t \in[0,T]})$, where $(\mathcal{F}_t)_t$ is the completion of $\big(\sigma\big(\pi(s) | 0 \leq s \leq t \big)\big)_t$ for a 
\[
\text{Poisson process } \pi=(\pi(t))_{t \in [0,T]} \text{ with intensity } \theta > 0.
\]
We call a two-dimensional stochastic process 
\[
(u(t))_{s \in [0,T)}=(\xi(t),\eta(t))_{t \in [0,T)}
\]
a \emph{control} if $\xi$ is progressively measurable and $\eta$ is predictable. We fix $T>0$ and $x \in \mathds{R}$; given a control $u$, the \emph{controlled process} at time $t \in [0,T)$ is given by 
\begin{equation} \label{EqCSDE}
X^u(t):=  x- \int_0^t \xi(s) ds -  \int^t_0\eta(s) d\pi(s)
\end{equation}
and the costs $u$ are given by
\begin{equation} \label{Costfunctional}
J(T,x,u):= \mathbb{E} \Big[ \int_0^T f(\xi(t),\eta(t),X^u(t)) dt \Big],
\end{equation}
where
\[
f:\mathds{R}^{n } \times \mathds{R}^{n} \rightarrow \mathds{R}, \quad f(\xi,\eta ,x) :=  \lambda \xi^2  + \gamma |\eta| +  \alpha x^2 \quad \text{for } \lambda,\gamma >0,  \alpha \geq 0.
\]
In view of Equation~\eqref{EqCSDE} (and somewhat unconventionally), we denote $\xi$ the \emph{absolutely continuous control} and $\eta$ the \emph{jump control} (note that the jump size and not the jump time is controlled).
We impose the following conditions.

\begin{definition}  \label{DefAdmStr}
Let $T>0$ and $x \in \mathds{R}$ be fixed and
$
u=(u(t))_{t \in [0,T)}=(\xi(t),\eta(t))_{t \in [0,T)}
$
be a control. \\
We call $u$ an
\emph{admissible control} and write $u \in \mathbb{A}(T,x)$ if 
\begin{equation} \label{EqIntegrability}
J(T,x,u)< \infty
\end{equation}
and
\begin{equation} \label{EqLiqConstraint}
\lim\limits_{t \rightarrow T-} X^u(t)= 0 \quad \text{a.s.}\footnote{The Terminal Constraint~\eqref{EqLiqConstraint} can be replaced by any other real number; this corresponds to a simple shift of the solution of the Optimization Problem~\eqref{EqValueFctAdv}.}
\end{equation}
\end{definition}

We readily observe that Definition~\ref{DefAdmStr} is satisfied if $\xi(t)\equiv \frac{x}{T}$ and $\eta(t)=0$ for all $t \in [0,T)$; hence $\mathbb{A}(T,x)\not= \emptyset$. The goal of this paper is to solve minimization problem
\begin{equation} \label{EqValueFctAdv}
v(T,x):=\inf\limits_{u \in \mathbb{A}(T,x)}  J(T,x,u). \tag{OPT}
\end{equation}

The key step towards the solution of the Optimization Problem~\eqref{EqValueFctAdv} is the derivation of a candidate 
$
w:(0,\infty) \times \mathds{R} \longrightarrow \mathds{R}_+
$
for the value function. Heuristic considerations suggest that $w$ should satisfy
the following HJB equation:

\begin{equation} \label{EqHJBAdv} \tag{$\text{HJB}$}
\begin{split}
\frac{\partial w}{\partial T} (T,x)& =\inf\limits_{ u=(\xi,\eta) \in \mathds{R} \times \mathds{R} }
	\Big[ \theta \big( w(T,x-\eta)-w(T, x) \big) - \frac{\partial w}{\partial x} (T,x) \xi + f(\xi,\eta ,x)\Big]
	\\
\lim\limits_{T\rightarrow 0+} w(T,x)&=\begin{cases}
	0  &\text{if } x=0 \\
	\infty &\text{else.} 
\end{cases}
\end{split}
\end{equation}

The singularity of the HJB equation is due to the Terminal Constraint~\eqref{EqLiqConstraint}; it renders the solution of the optimization problem non-standard. It turns out that we require the speed of convergence of $X^u(t)$ as $t \rightarrow T$ in order to conduct a verification argument using the HJB Equation~\eqref{EqHJBAdv}. However, the Optimization Problem~\eqref{EqValueFctAdv} is not linear-quadratic because of the cost term $\gamma |\eta|$; a closed form candidate for the solution of the problem (which turns into a qualitative bound for the speed of convergence of $X^u(t)$) is not readily available. In order to overcome this problem, we construct a candidate for the solution of~\eqref{EqValueFctAdv} in closed form; to this end, involved heuristic considerations which are based on the trade-off of the quadratic costs $\lambda \xi^2$ and the absolute value costs $\gamma |\eta|$ are required. We can reduce the complexity of the problem from solving the PDE~\eqref{EqHJBAdv} to solving a system of ODEs whose solution is explicitly available.

\section{Heuristic derivation of the value function} \label{SecAdvHeuristic}

Let us for now assume that $x>0$ and $\alpha  >0$. The results are symmetric and can directly be transferred to the case $x<0$; the case $\alpha  =0$ has a somewhat simpler structure and can be obtained by straightforward modifications or by taking the limit $\alpha  \rightarrow 0$ (cf.~Section~\ref{SubSecAdvPropertiesNoRisk}).

Given the cost term $\gamma |\eta|$, we can not expect the value function to be quadratic (or at least a quadratic polynomial). However, it should in sense be similar to a quadratic polynomial as we elaborate below.\footnote{This observation is also consistent with discrete-time constrained optimization problems with absolute value cost terms in addition to quadratic costs terms, see~\cite{Cui2012} and~\cite{Kratz2010} where the value function is piecewise a quadratic polynomial.}

We first recall the case $\gamma=0$ from~\cite{Kratz2012}. The value function $v$ and the optimal control $u^*_T=(\xi^*_T,\eta^*_T)\in \mathds{A}(T,x)$ are given by
\begin{equation} \label{EqOptNODP}
v(T,x)=C(T) x^2, \quad  \xi_T^*(t) = \frac{C(T-t)}{\lambda} X_T^*(t) , \quad \eta_T^*(t) = X_T^*(t-),
\end{equation}
where
\begin{equation}  \label{Eq1dimesionalvaluefct} 
C(T)
=\frac{\lambda\tilde{\theta}}{2} \coth \Big(\frac{\tilde{\theta} T}{2}  \Big) -\frac{\lambda \theta}{2}
\quad \text{for} \quad
\tilde{\theta}:=\sqrt{\theta^2+\tfrac{4 \alpha }{\lambda}}
\end{equation}
and $X^*=X^{u^*_T}$ is the controlled process. Hence the optimal control is of Markovian form; we will later see that this property is preserved for $\gamma>0$. The special case $\theta=0$ is the solution of a standard variational problem (see, e.g., \cite{Almgren2001}); in this case, the solution is given by
\begin{equation}\label{EqValFuncAdvBelow}
v(T,x)=C_0(T) x^2, \quad  \xi_T^*(t) = \frac{C_0(T-t)}{\lambda} X_T^*(t) \quad \text{for} \quad
C_0(T)
=\sqrt{\lambda \alpha} \coth \Big(\sqrt{\frac{\alpha}{\lambda}} T \Big) .
\end{equation}

We make the following heuristic observations about the structure of the value function and the optimal control which are partly inspired by the discrete time case.

\begin{itemize}
\item
For small initial states $x$, the absolute value costs $\gamma |\eta|$ outweigh the quadratic costs $\lambda \xi^2$. Hence, $\eta$ should be zero and the value function should satisfy
$v(T,x) = C_0(T) x^2$
(cf.~Equation~\eqref{EqValFuncAdvBelow}). Also the optimal absolute continuous control should be given as in~\eqref{EqValFuncAdvBelow}.
\item
It turns out that for large initial states $x$, the absolute value costs $\gamma |\eta|$ are more or less neglectable compared to the quadratic costs $\lambda \xi^2$. For that reason, the value function turns out to be a quadratic polynomial if $x$ is large enough.
\item
For intermediate states, we ``interpolate'' the value function in such a way that it is continuously differentiable.
\end{itemize}

We establish an interpolation method accomplishing this task: a candidate of the value function is identified in closed form and we call its representation as a ``polynomial'' with state-dependent coefficients a \emph{quasi-polynomial} in the following.

In Section~\ref{SecAdvHeuristicValuefunction}, we make the ansatz that the value function is a quasi-polynomial of degree two; we deduce a candidate for the optimal control (dependent on the coefficients of this quasi-polynomial) using the HJB Equation~\eqref{EqHJBAdv} and obtain a candidate for the boundary below which $\eta=0$ is optimal. In Section~\ref{SecAdvHeuristicDiff}, we derive a system ordinary differential equations for the coefficients of the candidate value function; the solution of this system is explicitly known.

\subsection{Heuristic derivation of the structure of the value function and the optimal strategy} \label{SecAdvHeuristicValuefunction}

We make the ansatz that a candidate $w$ for the value function
is a quasi-polynomial of degree two:
\begin{equation} \label{EqValueFuncPol1}
w(T,x)=\bar{C}_1(T,x) x^2+ \bar{C}_2(T,x) |x| + \bar{C}_3(T,x)
\end{equation}
for functions 
$\bar{C}_i:(0,\infty) \times \mathds{R} \rightarrow \mathds{R}$, $i=1,2,3$ (cf.~the discussion above). Obviously, the value function of every optimization problem can be written in this form. In the present case, this notation turns out to be helpful as it emphasizes the connection with linear-quadratic control problems. Instead of characterizing a candidate of the value function as a solution of the PDE~\eqref{EqHJBAdv}, the notation enables us to characterize it by a coupled system of ODEs for the coefficients $C_i$.

In the following we reflect on the optimal control. These considerations yield
possible properties of the coefficients of $w$ as in Equation~\eqref{EqValueFuncPol1}.
Assume that $u_T^*=(\xi_T^*,\eta_T^*)\in \mathbb{A}(T,x)$ is the optimal control.
For small states $x$, the absolute value costs $\gamma |\eta|$ are larger than the quadratic costs $\lambda \xi^2$. Therefore, there should be a boundary $\beta:(0,\infty)\rightarrow \mathds{R}_+$ such that 
\begin{equation} \label{EqAdvEta*Below1}
\eta^*_T(0):=\eta^*(T,x)=0 \quad \text{for } |x| \leq \beta(T).
\end{equation}
For larger states, the absolute value costs are smaller than the quadratic costs; therefore we expect
that the control $\eta^*_T$ is such that a jump of $\pi$ decreases the controlled process to the level $\beta$ where
further controls $\eta>0$ are too costly:
\begin{equation} \label{EqAdvHeuristicetaboundary}
\eta^*_T(0)=\eta^*(T,x) =
\sgn(x) \big(|x|-\beta(T)\big) \quad \text{for } |x| > \beta(T).
\end{equation}
As $\eta^*(T,x)=0$ for $|x| \leq \beta(T)$, the optimal absolutely continuous control (and the value function) below the boundary
should be the one given by Equation~\eqref{EqValFuncAdvBelow}. In other words, we expect the coefficients of $w$ as in Equation~\eqref{EqValueFuncPol1} 
to fulfill
\[
\bar{C}_1(T,x) =C_0(T), \quad \bar{C}_2(T,x)=\bar{C}_3(T,x)=0 \quad \text{for } T>0\text{ and } |x| \leq \beta(T).
\]
We illustrate these heuristics in the left and the middle picture of Figure~\ref{FigHeu}. 

Let us now assume that these considerations are true. We hope that the value function is differentiable
on $(0,\infty) \times \mathds{R}$. We use this property for the proof of
the verification theorem 
via the HJB Equation~\eqref{EqHJBAdv}
(cf.~Theorem~\ref{TheoremAdvOptLiq} below). 
A necessary condition for continuity of $w$
at $x=\beta(T)$ is
\[
\bar{C}_1(T,\beta(T)) =C_0(T),  \quad \bar{C}_2(T,\beta(T))= \bar{C}_3(T,\beta(T))=0.
\]
In order to be differentiable at $x=\beta(T)$, the right-hand and the left-hand partial derivative
with respect to $x$  
must be equal, i.e., for $T>0$, $x=\beta(T)$,
\begin{equation} \notag
2C_0(T) x = 2\bar{C}_1(T,x) x+ \bar{C}_2(T,x) 
+ x^2 \frac{\partial\bar{C}_1}{\partial{x}}(T,x) + x\frac{\partial\bar{C}_2}{\partial{x}}(T,x) 
+ \frac{\partial\bar{C}_3}{\partial{x}}(T,x),
\end{equation}
and therefore a necessary condition for differentiability is 
\begin{equation*}
\beta(T)^2 \frac{\partial\bar{C}_1}{\partial{x}}(T,\beta(T)) + \beta(T)\frac{\partial\bar{C}_2}{\partial{x}}(T,\beta(T)) 
+ \frac{\partial\bar{C}_3}{\partial{x}}(T,\beta(T))=0.
\end{equation*}

We make the educated guess that this condition also holds for $x \not= \beta(t)$ (note that it certainly holds for small $x < \beta(t)$ and large $x$ as in these cases the coefficients are constant in $x$):
\begin{equation} \label{EqAdvMainLemmaDiffEq}
x^2 \frac{\partial\bar{C}_1}{\partial{x}}(T,x) + |x| \frac{\partial\bar{C}_2}{\partial{x}}(T,x) 
+ \frac{\partial\bar{C}_3}{\partial{x}}(T,x) =0
\end{equation}
for $T>0$, $x \in \mathds{R}$ and thus
\begin{equation} \label{EqAdvAbleitungHeuristic}
\frac{\partial w}{\partial{x}}(T,x) = 2 \bar{C}_1(T,x) x +\sgn(x) \bar{C}_2(T,x).
\end{equation}

\begin{figure}
\centering
\begin{tabular}{lll}
\begin{tabular}{l}
\begin{overpic}[height=2.7cm, width=4.5cm]{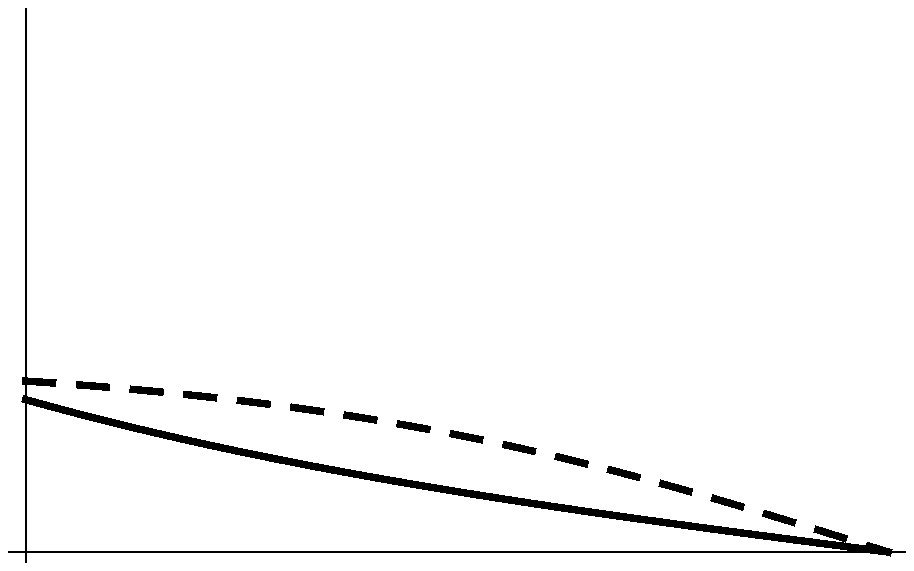}
\put(93,-3.5){\scriptsize{$T$}}
\put(5,23){\scriptsize{$\beta(T\!-\!\cdot)$}}
\put(0,60){\scriptsize{State}}
\end{overpic}
\end{tabular} &
\begin{tabular}{l}
\begin{overpic}[height=2.7cm, width=4.5cm]{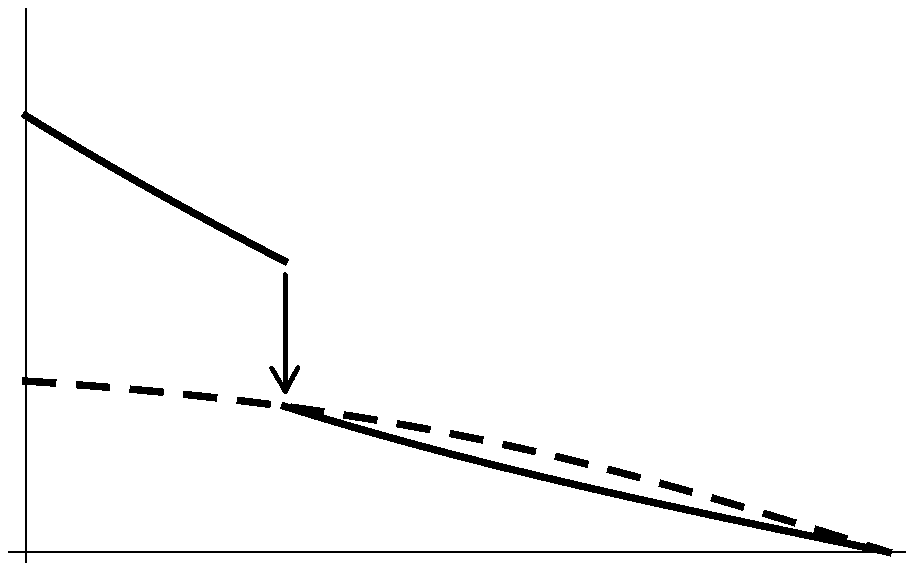}
\put(93,-3.5){\scriptsize{$T$}}
\put(5,12){\scriptsize{$\beta(\!T-\!\cdot)$}}
\put(29,-2.5){\scriptsize{$\tau$}}
\put(0,60){\scriptsize{State}}
\end{overpic}
\end{tabular} &
\begin{tabular}{l}
\begin{overpic}[height=2.7cm, width=4.5cm]{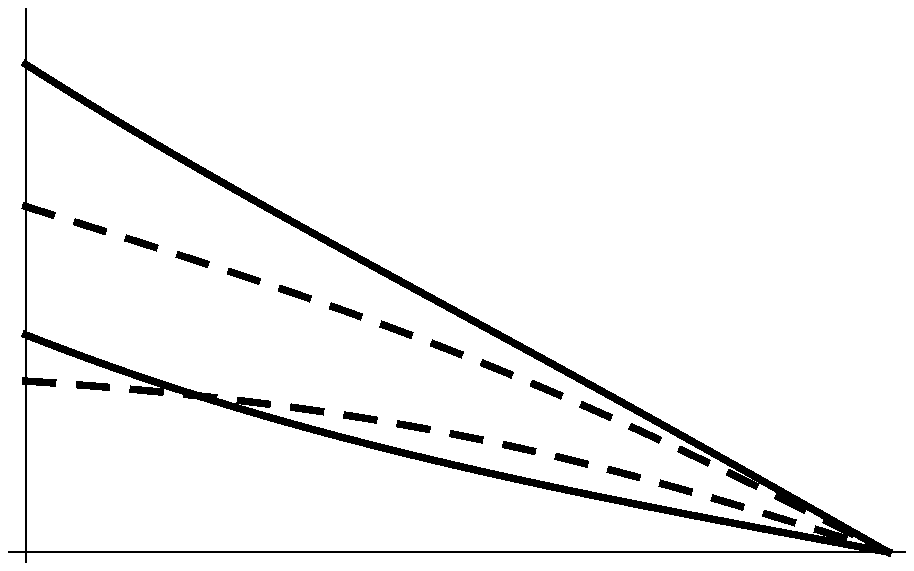}
\put(93,-3.5){\scriptsize{$T$}}
\put(5,11){\scriptsize{$\beta(T\!-\!\cdot)$}}
\put(5,25){\scriptsize{$\bar{X}(T\!-\!\cdot,0)$}}
\put(20,-2.5){\scriptsize{$s$}}
\put(0,60){\scriptsize{State}}
\end{overpic}
\end{tabular}
\end{tabular}\caption{Illustrations of the heuristics about the optimal control. In all pictures, the (lower) dashed line denotes the time dependent boundary $\beta(T-\cdot)$ below which $\eta^*_T=0$. Solid lines denote realized trajectories of the controlled process. The left picture shows the scenario where the initial state is below $\beta$. The middle picture shows a state which is initially above $\beta$. In the displayed scenario, the Poisson process jumps at time $\tau$. The third picture shows the optimal control for two states when the Poisson process never jumps. For the smaller initial state, the controlled process crosses the boundary $\beta(T-\cdot)$ at time $s$, i.e., for $t\leq s$ the smaller solid line represents the function $\bar{X}(T-\cdot,T-s)$. The larger trajectory does not cross $\beta$ during the entire horizon; it is above the function $\bar{X}(T-\cdot,0)$ which is displayed by the second dashed line.
} \label{FigHeu}
\end{figure}

We close the section by finding a candidate for the optimal absolute continuous control and a formula for $\beta$ (dependent on the coefficients of $w$). We assume from now on that the value function is given by $w$ as in Equation~\eqref{EqValueFuncPol1}, that all partial derivatives exist, that $\frac{\partial w}{\partial x}$ is given as in Equation~\eqref{EqAdvAbleitungHeuristic} and that the optimal jump control $\eta^*$ is given by Equations~\eqref{EqAdvEta*Below1} and~\eqref{EqAdvHeuristicetaboundary}.

We consider the HJB Equation~\eqref{EqHJBAdv}; for $T>0$ and $x \in \mathds{R}$, we minimize the function
\begin{align*}
h(T,x,\xi , \eta) &:=   \theta \big( w(T,x- \eta)-w(T, x) \big) - \frac{\partial w}{\partial x} (T,x) \xi 
	+\lambda \xi^2 + \gamma |\eta| + \alpha x^2 \\
&=   \theta w(T,x-\eta) + \gamma |\eta| - \theta  w(T, x)  + \alpha x^2 + \lambda \Big( \frac{\frac{\partial w}{\partial x} (T,x)}{2\lambda}- \xi \Big)^2 -
	\frac{\frac{\partial w}{\partial x} (T,x)^2}{4\lambda}
\end{align*}
in $(\xi,\eta)$.
Using Equation~\eqref{EqAdvAbleitungHeuristic}, we obtain that $h$ is minimal for
\begin{equation} \label{EqAdvHeurxi*}
\xi^*:=\xi^*_T(0):=\xi^*(T,x):=\frac{2\bar{C}_1(T,x) x + \sgn(x) \bar{C}_2(T,x)}{ 2\lambda}
\end{equation}
and for $\eta^*$ such that 
$
\bar{h}(\eta ) := \theta w(T, x- \eta ) + \gamma |\eta |
$
is minimal.
For $x > \beta(T)$, Equation~\eqref{EqAdvHeuristicetaboundary} suggests that this
should be the case for
\begin{equation} \label{EqAdvHeurEta1}
\eta^*:=\eta_T(0) = \eta^*(T,x)= x-\beta(T).
\end{equation}
Furthermore for $\eta^*>0$, 
$
\bar{h}'(\eta^*)
	=- \theta \frac{\partial w}{\partial x}(T,x-\eta^*) +\gamma =0
$
if and only if
$
\frac{\partial w}{\partial x}(T,x-\eta^*) =\frac{\gamma}{\theta}.
$
Combining these observations, we obtain the
following candidate for the boundary $\beta(T)$:
\begin{equation} \label{EqCandidateb}
\beta(T)=\frac{\gamma}{2 \theta C_0(T)}.
\end{equation}
Equations~\eqref{EqAdvHeurxi*}, \eqref{EqAdvHeurEta1} and~\eqref{EqCandidateb} suggest that the optimal control $u^*_T=(\xi^*_T,\eta^*_T)$ is of Markovian form with
\begin{align}
\xi^*_T(t)&:=\xi^*(T-t,y)=\frac{2\bar{C}_1(T-t,y) y + \sgn(y) \bar{C}_2(T-t,y)}{ 2\lambda}  \label{EqAdvHeurxi*1}\\
\eta^*_T(t)&:= \eta^*(T-t,y)= \begin{cases} 
\sgn(y) \big(|y|-\frac{\gamma}{2 \theta C_0(T-t)}\big) &\text{ if } y > \frac{\gamma}{2 \theta C_0(T-t)} \\
0 &\text{ else},
\end{cases} \label{EqAdvHeureta*1}
\end{align}
where $y$ is the state of the controlled process at time $t-$.

\subsection{Ordinary differential equations for the coefficients} \label{SecAdvHeuristicDiff}

As indicated before, we have written the candidate of the value function $w$ as a quasi-polynomial in order to reduce the problem of solving a PDE for $w$ to solving ODEs for the $C_i$'s.

For now, we consider the case $x>\beta(T)$. Given the assumption that $w$ as in Equation~\eqref{EqValueFuncPol1}
solves the HJB Equation~\eqref{EqHJBAdv} with minimizer $(\xi^*,\eta^*)$ for $\xi^*$ and $\eta^*$
as in Equations~\eqref{EqAdvHeurxi*} and \eqref{EqAdvHeurEta1}, respectively, we obtain
\[
\frac{\partial w}{\partial T} (T,x) = h(T,x,\xi^*,\eta^*). 
\]
Provided that Equations~\eqref{EqAdvMainLemmaDiffEq} and~\eqref{EqCandidateb} hold, this implies
\begin{align}
&\frac{\partial \bar{C}_1}{\partial T}(T,x) x^2
	+\frac{\partial \bar{C}_2}{\partial T}(T,x) x+\frac{\partial \bar{C}_3}{\partial T} (T,x) \notag \\
&\qquad =\Big(\alpha- \frac{\bar{C}_1(T,x)^2}{\lambda} - \theta  \bar{C}_1(T,x) \Big) x^2
  	+ \Big( \gamma -\bar{C}_2(T,x) \Big(\frac{\bar{C}_1(T,x)}{\lambda}+\theta \Big)  \Big) x - \theta \bar{C}_3(T,x) - \frac{\gamma^2}{4 \theta C_0(T)} - \frac{\bar{C}_2(T,x)^2}{4\lambda}. \notag
\end{align}
We hence expect that for $|x|>\beta(T)$, the coefficients of the 
candidate value function fulfill the ordinary
differential equations
\begin{align} 
\frac{\partial \bar{C}_1}{\partial T}(T,x) &= 
	\alpha - \frac{\bar{C}_1(T,x)^2}{\lambda} - \theta \bar{C}_1(T,x), \label{DiffEqAdvC1alt}  \\
\frac{\partial \bar{C}_2}{\partial T}(T,x) &= \gamma
	- \bar{C}_2(T,x) \Big(\frac{\bar{C}_1(T,x)}{\lambda}+\theta \Big), \label{DiffEqAdvC2alt} \\
\frac{\partial \bar{C}_3}{\partial T} (T,x)&=
	-\theta \bar{C}_3(T,x) - \frac{\gamma^2}{4 \theta C_0(T)} - \frac{\bar{C}_2(T,x)^2}{4\lambda} \label{DiffEqAdvC3alt} 
\end{align}
and that for $x=\beta(T)$,
\begin{equation} \label{AdvDiffeqAnfangsbedalt}
\bar{C}_1(T,x)=C_0(T), \quad \bar{C}_2(T,x)=\bar{C}_3(T,x)=0.
\end{equation}

In order to transfer the Ordinary Differential Equations~\eqref{DiffEqAdvC1alt},
\eqref{DiffEqAdvC2alt}, \eqref{DiffEqAdvC3alt} and the Conditions~\eqref{AdvDiffeqAnfangsbedalt}
into initial value problems, we consider the scenario where the Poisson process never jumps (i.e., $\pi(T)=0$). Let $T>0$ and $x \geq \beta(T)$; assume further that there exists a first time $s \in [0,T)$ such that
the process controlled by the candidate optimal control $u^*_T$ (cf.~Equations~\eqref{EqAdvHeurxi*1} and~\eqref{EqAdvHeureta*1}), $X^*_T$, crosses the boundary $\beta$, i.e., 
\[
X^*_T(s)=\beta(T-s).
\]
We define
\begin{equation*}
\bar{X}(T,S):=x, \quad \text{where } S:=T-s;
\end{equation*}
in other words, the \emph{deterministic} function $\bar{X}(T-\cdot,S):[0,s]=[0,T-S]\rightarrow \mathds{R}_+$ satisfies 
\begin{equation} \label{EqbarXandX*}
\bar{X}(T-t,S) =X^*_T(t) \quad \text{for } t\in [0,T-S] \text{ if } \pi(T-S)=0.
\end{equation} Note also that this implies (cf.~Equation~\eqref{EqCandidateb})
\begin{equation*}
\bar{X}(T,T)=\beta(T)=\frac{\gamma}{2 \theta C_0(T)}.
\end{equation*}
We illustrate the functions $\bar{X}$ in the third picture of Figure~\ref{FigHeu}.

We now modify the notation for the coefficients of $w$ in the following way. Let $T>0$, $x\geq \beta(T)$ and $S$ as before. We write
\begin{equation} \notag
C_1(T,S):=\bar{C}_1(T,x), \quad C_2(T,S):=\bar{C}_2(T,x), \quad C_3(T,S):=\bar{C}_3(T,x).
\end{equation}
Thus we expect that for fixed $S>0$, the new coefficients $C_1,C_2$ and $C_3$ solve the following initial value problems on $[S,\infty)$:
\begin{align} 
\frac{\partial C_1}{\partial T}(T,S) &= 
	\alpha- \frac{C_1(T,S)^2}{\lambda}- \theta C_1(T,S) , \quad 
C_1(S,S)=C_0(S), \label{DiffEqAdvC1}\\
\frac{\partial C_2}{\partial T}(T,S) &= 
	\gamma - C_2(T,S) \Big(\frac{C_1(T,S)}{\lambda}+\theta \Big) 
	,\quad 
C_2(S,S)=0,  \label{DiffEqAdvC2}  \\
\frac{\partial C_3}{\partial T} (T,S)&=
	-\theta C_3(T,S) - \frac{\gamma^2}{4 \theta C_0(T)} 
	- \frac{C_2(T,S)^2}{4\lambda}, \quad
C_3(S,S)=0; \label{DiffEqAdvC3}
\end{align}
furthermore, we expect $\bar{X}$ to solve the initial value problem
(cf.~Equations~\eqref{EqCSDE}, \eqref{EqAdvHeurxi*} and~\eqref{EqbarXandX*})
\begin{equation}\label{DiffEqAdvXbar}
\frac{\partial \bar{X}}{\partial T}(T,S) = \xi^*(T,\bar{X}(T,S))=
	\frac{C_1(T,S)}{\lambda}\bar{X}(T,S)+ \frac{C_2(T,S)}{2\lambda},\quad 
\bar{X}(S,S)=\frac{\gamma}{2\theta C_0(S)}.
\end{equation}

\section{The candidate value function} \label{SecPropValAdv}

Following the heuristic considerations of Section~\ref{SecAdvHeuristic}, we define the candidate value function via $w$ by
\begin{equation} \label{EqAdvCandidateValueFct}
w(T,x):=
C_1(T,g(T,x))x^2 + C_2(T,g(T,x)) |x| + C_3(T,g(t,x)),
\end{equation}
where
$C_1$, $C_2$ and $C_3$ are the solutions of the Initial Value Problems~\eqref{DiffEqAdvC1}, 
\eqref{DiffEqAdvC2} and \eqref{DiffEqAdvC3}, respectively; the function 
$g :(0,\infty)\times \mathds{R} \rightarrow [0,T]$ is given explicitly for 
$|x| \in [0, \frac{\gamma}{2 \theta C_0(T)}] \cup [\bar{X}(T,0),\infty)$ by 
\begin{equation} \label{EqAdvDefg1}
g(T,x) :=
\begin{cases}
T &\text{if } |x| \in [0, \frac{\gamma}{2 \theta C_0(T)}]\\
0 &\text{if } |x| \in [\bar{X}(T,0),\infty),
\end{cases} 
\end{equation}
and implicitly for $|x|  \in (\frac{\gamma}{2\theta C_0(T)},\bar{X}(T,0))$ by 
\begin{equation} \label{EqAdvDefg2}
\bar{X}(T,g(T,x)) = |x|,
\end{equation}
where $\bar{X}$ is the solution of the Initial Value Problem~\eqref{DiffEqAdvXbar},
i.e., $g(T,\cdot)$ is the inverse function of $\bar{X}(T,\cdot)$.
The fact that such an inverse function exists
(and that hence the function $w$ defined by Equation~\eqref{EqAdvCandidateValueFct}
is well-defined) is verified in Section~\ref{SubsecAdvExplicit} below. 

The remainder of the section is structured as follows. In Section~\ref{SubsecAdvExplicit}, we 
compute the solutions of the Initial Value 
Problems~\eqref{DiffEqAdvC1}, \eqref{DiffEqAdvC2}, \eqref{DiffEqAdvC3} and
\eqref{DiffEqAdvXbar} in closed form. This enables us to prove that $w$ is well-defined. 
In Section~\ref{SubSectAdvDiff}, we prove that
$w$ is continuously differentiable and strictly convex in $x$. Finally, we show that $w$ is a classical solution of
the HJB Equation~\eqref{EqHJBAdv} in Section~\ref{SubSectAdvHJB}.

In Sections~\ref{SecPropValAdv} and~\ref{SecVerificationAdv}, we assume that
$\alpha >0$. The significantly simpler case $\alpha =0$ is in particular discussed in Section~\ref{SubSecAdvPropertiesNoRisk}.

\subsection{Closed form solutions for the coefficients and the trading trajectory} \label{SubsecAdvExplicit}

In order that the candidate value function is well-defined, the function $g$ must be well-defined. Therefore,
we require that $\bar{X}$ is strictly monotone in $S$ (cf.~Equation~\eqref{EqAdvDefg2}). We
prove this by directly computing the partial derivative of $\bar{X}$ with respect to $S$
and show that it is strictly negative. 

We start by giving closed form solutions for the value function coefficients $C_1$, $C_2$, $C_3$ and for $\bar{X}$, 
which follow from the Initial Value Problems~\eqref{DiffEqAdvC1}, \eqref{DiffEqAdvC2}, \eqref{DiffEqAdvC3} and 
\eqref{DiffEqAdvXbar}, respectively. Proposition~\ref{PropAdvExplicitC} treats the case $S \in (0,T)$, the case $S=0$ is treated separately afterwards.

\begin{prop} \label{PropAdvExplicitC}
Let $T>0$ and $S \in (0,T)$ and define
\begin{equation}\label{EqAdvKappas}
\mu(S):=\frac{2C_0(S)+ \theta \lambda}{\tilde{\theta} \lambda}, \quad
\kappa(S):= \arcoth(\mu(S));
\end{equation}
recall that
$\tilde{\theta}=\sqrt{\theta^2+\tfrac{4 \alpha }{\lambda}}
$.
Then
\begin{align}
C_1(T,S)&= \frac{\lambda \tilde{\theta}}{2} \coth \big(\frac{\tilde{\theta}}{2}(T-S) 
	+\kappa(S)\big) - \frac{\lambda \theta}{2} >0, \label{EqAdvC1} \\
C_2(T,S)&= \frac{ \gamma}{2 \alpha} 
	\Big( \!-\theta \lambda +  \frac{ \lambda
	\tilde{\theta} \big( \sinh \big(\frac{\tilde{\theta}}{2}(T\! -\! S)\big) 
	 +  \mu(S) \cosh\big(\frac{\tilde{\theta}}{2}(T\! -\! S)\big)\big)
	 -  2 C_0(S) \exp \big( \frac{\theta}{2}(S\!-\!T)\big)}
	{ \mu(S) \sinh\big(\frac{\tilde{\theta}}{2}(T\! -\! S) \big)
	+ \cosh\big(\frac{\tilde{\theta}}{2}(T \!-\! S) \big)} \Big) \label{EqAdvC2} >0,  \\
C_3(T,S) &= - \int_S^T \exp \big( \theta (r-S) \big) 
	\Big(\frac{\gamma^2}{4\theta C_0(T+S-r)} + \frac{C_2(T+S-r,S)^2}{4 \lambda} \Big) dr
<0, \label{EqAdvC3}\\
\bar{X}(T,S)&= \Big(\frac{\gamma \mu(S)}{2 \theta C_0(S)} 
	+\frac{\theta \gamma}{2 \tilde{\theta} \alpha} \Big)
	\sinh\big(\frac{\tilde{\theta}}{2}(T-S) \big)\exp\big(\frac{\theta}{2}(S-T) \big) \notag \\*
& \qquad + \Big(\frac{\gamma }{2 \theta C_0(S)} 
	+\frac{\gamma}{2 \alpha} \Big)
	\cosh\big(\frac{\tilde{\theta}}{2}(T-S) \big)\exp\big(\frac{\theta}{2}(S-T) \big) 
	- \frac{\gamma}{2 \alpha}.  \label{EqAdvXbar}
\end{align}
\end{prop}

\begin{proof}
Note first that for $S\in (0,T)$,
$
C_0(S) > \sqrt{\alpha \lambda}
$
by Equation~\eqref{EqValFuncAdvBelow}. Therefore, 
\begin{equation*}
\lambda \tilde{\theta} = \sqrt{ \lambda^2 \theta^2 + 4 \alpha \lambda}
	< \lambda \theta +2\sqrt{\alpha \lambda} \leq \lambda \theta + 2 C_0(S);
\end{equation*}
thus
$
\mu(S) >1
$
and $\kappa(S)$ as in Equation~\eqref{EqAdvKappas} is well-defined for all $S \in (0,T]$.

Equation~\eqref{EqAdvC1} follows directly as~\eqref{DiffEqAdvC1} is an initial value problem for a Riccati equation with constant coefficients. \eqref{DiffEqAdvC2}, \eqref{DiffEqAdvC3} and~\eqref{DiffEqAdvXbar} are initial value problems for inhomogeneous linear differential equations whose solutions are readily computed (note that $\frac{1}{C_0(\cdot)}$ is continuous). Elementary but tedious algebraic manipulations yield Equations~\eqref{EqAdvC2}, \eqref{EqAdvC3} and~\eqref{EqAdvXbar}.\footnote{Detailed proofs for these and later algebraic manipulations can be found in~\cite{Kratz2011}. \label{FootnoteManip}} 
\end{proof}

As $C_0(0)$ is not defined and $\lim_{S\rightarrow 0+} C_0(S)=\infty$, $C_1(T,S)$, $C_2(T,S)$, $C_3(T,S)$ and $\bar{X}(T,S)$ ($T >0$) as in Equations~\eqref{EqAdvC1}~-~\eqref{EqAdvXbar} are not defined for $S=0$ . However, we require these objects for the
definition
of the candidate value function. We define $C_i(T,0):= \lim_{S\rightarrow 0+} C_i(T,S)$ and $\bar{X}(T,0):=\lim_{S\rightarrow 0+} \bar{X}(T,S)$.
Elementary considerations verify that these limits exist and are given by
\begin{align}
C_1(T,0)&=
	\frac{\lambda \tilde{\theta}}{2} \coth \big(\frac{\tilde{\theta}}{2}T\big) 
	- \frac{\lambda \theta}{2} >0,	\label{EqAdvC1T}  \\
C_2(T,0)& =\frac{ \gamma \lambda}{2 \alpha} 
	\Big( \tilde{\theta} \coth \big(\frac{\tilde{\theta}}{2}T  \big) 
	-\tilde{\theta} 
	\frac{\exp \big(- \frac{\theta}{2} T\big)}
	{\sinh\big(\frac{\tilde{\theta}}{2} T  \big)}
	-\theta \Big)  \label{EqAdvC2T} >0,  \\
C_3(T,0)&=- \int_0^T \exp \big( -\theta r \big) 
	\Big(\frac{\gamma^2}{4 \theta C_0(T-r)} + \frac{C_2(T-r,0)^2}{4 \lambda} \Big)  dr <0, 
	\label{EqAdvC3T}  \\
\bar{X} (T,0)&= \Big(\frac{\gamma}{\tilde{\theta}\theta \lambda} 
	+ \frac{ \theta \gamma}{2 \tilde{\theta} \alpha} \Big)
	\sinh \big(\frac{\tilde{\theta}}{2}T  \big)
	\exp\big(-\frac{\theta}{2}  T \big) \! + \!\frac{\gamma}{2 \alpha}
	\cosh \big(\frac{\tilde{\theta}}{2}T \big) \exp\big(-\frac{\theta}{2}  T \big)
	-\frac{\gamma}{ 2 \alpha}. \label{EqAdvXbarT}
\end{align}

It follows immediately from the Formulae~\eqref{EqAdvC1} - \eqref{EqAdvXbar} that both the coefficients of the candidate value function and the trajectories $\bar{X}$ are continuously differentiable, even analytic.

\begin{cor} \label{CorAdvContDiff2Times}
Let $T>0$ and $S\in (0,T)$. Then $C_1$, $C_2$, $C_3$ and $\bar{X}$ have continuous partial derivatives
of any order in $(T,S)$.
\end{cor}

In order to show that the function $g$ given in Equations~\eqref{EqAdvDefg1} and~\eqref{EqAdvDefg2} and thus the
candidate value function is well-defined, we need monotonicity of the function $\bar{X}(t,\cdot)$. Given Equation~\eqref{EqAdvXbar}, we can directly compute the partial derivative of $\bar{X}$ with respect to the second argument and verify that $\bar{X}(T,\cdot)$ is strictly decreasing in $(0,T)$ as required. We can hence prove the central result of this section.

\begin{prop} \label{PropWellDefined}
The candidate value function $w$ given by Equation~\eqref{EqAdvCandidateValueFct} is well-defined.
\end{prop}

\begin{proof}
By the discussion above, $\bar{X}(T,\cdot)$ possesses an inverse on $(0,T)$. Thus, $g$ as in 
Equations~\eqref{EqAdvDefg1} and~\eqref{EqAdvDefg2} is well-defined on $(0,\infty) \times \mathds{R}$ (cf.~also Equation~\eqref{EqAdvXbarT}). Furthermore,
$g(T,x)=g(T,|x|) \in [0,T]$ for all $T>0$, $x\in \mathds{R}$, and the functions $C_i$ ($i=1,2,3$) are well-defined for $T>0$, $S \in [0,T]$.
\end{proof}

\subsection{Differentiation and convexity of the candidate value function} \label{SubSectAdvDiff}

The goal of this section is to show that the candidate value function $\tilde{w}$ is continuously differentiable.  We are also able to compute the partial derivatives 
$\frac{\partial w}{\partial T}$, $\frac{\partial w}{\partial x}$ and $\frac{\partial^2 w}{\partial x^2}$.
It turns out that $\frac{\partial w}{\partial x}$
has indeed the form we assumed in the heuristics in 
Section~\ref{SecAdvHeuristicValuefunction} (cf.~Equation~\eqref{EqAdvAbleitungHeuristic}).
We also show that $\tfrac{\partial^2 w}{\partial x^2} > 0$ and hence that $w$ is strictly convex in $x$.
We first state the main theorem.

\begin{theorem} \label{TheoremAdvContDiffw}
The candidate value function $\tilde{w}$ given via $w$ as in Equation~\eqref{EqAdvCandidateValueFct} is continuously differentiable on $(0,\infty) \times \mathds{R}$. Furthermore, we have
\begin{align}
\frac{\partial w}{\partial x} (T,x) &= 2 C_1(T,g(T,x)) x + \sgn(x) C_2(T,g(T,x)),  \label{EqAdvdwdx}\\
\frac{\partial w}{\partial T} (T,x) &=
\begin{cases}
C'_0(T) x^2 &\text{if } x\leq\frac{\gamma}{2 \theta C_0(T)} \\ 
\frac{\partial C_1}{\partial T}(T,g(T,x)) x^2 + \frac{\partial C_2}{\partial T}(T,g(T,x))|x| 
	+ \frac{\partial C_3}{\partial T}(T,g(T,x)) & \text{else.}	
\end{cases}
\end{align}
\end{theorem}

Before we proceed with the proof of Theorem~\ref{TheoremAdvContDiffw}, we show 
that $g$ is continuously differentiable. We prove this in the following lemma and compute
the partial derivatives of $g$.

\begin{lem} \label{LemAdvDiffg}
Let $V:=\{(T,x)\in \mathds{R}^2| T>0, \, x \in (\bar{X}(T,T),\bar{X}(T,0))\}$. 
Then $g$ given by Equations~\eqref{EqAdvDefg1} and~\eqref{EqAdvDefg2} is continuously differentiable on $V$ and
\begin{equation}
\frac{\partial g}{\partial x}(T,x)= \frac{1}{\frac{\partial \bar{X}}{\partial S} (T,g(T,x))}, \quad
\frac{\partial g}{\partial T}(T,x)= 
-\frac{\frac{\partial \bar{X}}{\partial T} (T,g(T,x))}{\frac{\partial \bar{X}}{\partial S} (T,g(T,x))}.  \label{EqAdvgdx}
\end{equation}
\end{lem}

\begin{proof}
Consider the open set $U:=\{(T,S) \in \mathds{R}^2 | T>0, S \in (0,T) \}$ and the 
continuously differentiable function $h:\mathds{R}^2 \longrightarrow \mathds{R}^2$, $h(T,S):= (T,\bar{X}(T,S))^\top$ (cf. 
Corollary~\ref{CorAdvContDiff2Times}). Then for all $(T_0,S_0)\in U$, the Jacobian matrix of $h$ is given by
\begin{equation*}
Dh(T_0,S_0):=\begin{pmatrix}
1 && \frac{\partial \bar{X}}{\partial T} (T_0,S_0) \\
0 && \frac{\partial \bar{X}}{\partial S} (T_0,S_0) 
\end{pmatrix}.
\end{equation*} 
Note that $\frac{\partial \bar{X}}{\partial S} (T_0,S_0) \not= 0$ as $\bar{X}(T_0, \cdot)$ is strictly decreasing. 
Therefore $\det Dh(T_0,S_0) = \frac{\partial \bar{X}}{\partial S} (T_0,S_0) \not= 0,$ and we can apply the inverse function theorem. Thus, the inverse function $h^{-1}$ of $h$ exists locally and is locally continuously differentiable. However, on $\tilde V:=h(U)=\{(T,X)\in \mathds{R}^2| T>0,\, x \in (\bar{X}(T,T),\bar{X}(T,0))\}$, we have $h^{-1}=\tilde g$ for
$
\tilde g (T,x):=(T,g(T,x))^\top
$
globally, and therefore $g$ is continuously differentiable on $V$. 

The first equation in~\eqref{EqAdvgdx} follows from Equation~\eqref{EqAdvDefg2}
immediately; the second equation follows by differentiating with respect to $T$ 
on either side of Equation~\eqref{EqAdvDefg2}.
\end{proof}

It follows directly that 
$w$ is continuous on $(0,\infty) \times \mathds{R}$.
We can now prove the main building block of the proof of Theorem~\ref{TheoremAdvContDiffw}; we show that
\begin{equation} \notag
\frac{\partial C_1}{\partial S} (T,S) \bar{X}(T,S)^2 +
\frac{\partial C_2}{\partial S} (T,S) \bar{X}(T,S) + \frac{\partial C_3}{\partial S} (T,S)=0.
\end{equation}
This is equivalent to Equation~\eqref{EqAdvMainLemmaDiffEq} in the heuristics of 
Section~\ref{SecAdvHeuristicValuefunction} and was the key step towards the derivations of the 
differential equations for the coefficients $C_1$, $C_2$ and $C_3$. Therefore, the
construction of the candidate value function relies strongly on this property. 

The second part of the lemma is a similar assertion which is useful for the computation of 
the second order derivative
$\frac{\partial^2 w}{\partial x^2}$ and hence for the proof of the convexity 
of $w$ below.

\begin{lem} \label{LemAdvMainDiff}
Let $T>0$ and $S \in (0,T]$. Then
\begin{enumerate}
\item[(i)]
\begin{equation} \notag
a(T,S):=\frac{\partial C_1}{\partial S} (T,S) \bar{X}(T,S)^2 +
\frac{\partial C_2}{\partial S} (T,S) \bar{X}(T,S) + \frac{\partial C_3}{\partial S} (T,S)=0,
\end{equation}
\item[(ii)]
\begin{equation} \notag
b(T,S):=2 \frac{\partial C_1}{\partial S} (T,S) \bar{X}(T,S) +
\frac{\partial C_2}{\partial S} (T,S) =0.
\end{equation}
\end{enumerate}
\end{lem}

\begin{proof}
\begin{enumerate}
\item[(i)]
We first show that
\begin{equation} \label{EqAdvLimitagegens}
a(S,S)=\lim\limits_{T\rightarrow S+} a(T,S) = 0
\end{equation}
by computing the derivatives of $C_1$, $C_2$ and $C_3$ with respect to $S$.
Secondly, we verify that
\begin{equation} \notag
\frac{\partial a}{\partial T}(T,S)= \theta a(T,S). 
\end{equation}
Hence, there exists a constant $C(S)$ ($S \in (0,\infty)$) such that for all $T \in (S,\infty)$,
$
a(T,S)=C(S) \exp(-\theta T).
$
Equation~\eqref{EqAdvLimitagegens} implies 
$
0=\lim_{T \rightarrow S+} a(T,S) =C(S) \exp(-\theta S),
$
i.e., $C(S)=0$, finishing the proof of~(i).
\item[(ii)]
We proceed similarly as in the proof of~(i) and show $
\lim_{T \rightarrow S+} b(T,S) = 0
$ and
$\frac{\partial b}{\partial T} (T,S) 
= \big( \frac{C_1(T,S)}{\lambda} + \theta \big) b(T,S)$, from which the assertion follows.
\end{enumerate}
\end{proof}

We are now able to prove that $w$ is continuously differentiable.

\begin{proof}[Proof of Theorem~\ref{TheoremAdvContDiffw}]
Let $T>0$ and $x \in \mathds{R}$. We only treat the case $x \geq 0$. The case $x <0$ follows accordingly. We
consider four cases separately.  
\begin{enumerate}
\item[(i)] $x <\frac{\gamma}{2 \theta C_0(t)}$.
In this case we have $g(v,y)=v$ for all $(v,y)$ in some neighborhood of $(T,x)$
(cf.~Equation~\eqref{EqAdvDefg1}). Therefore, 
$
w(v,y)=C_0(v) y^2
$ 
and consequently (recall that $C_1(T,T)=C_0(T)$, $C_2(T,T)=0$)
\begin{equation} 
\frac{\partial w}{\partial x} (T,x)  = 2 C_0(T) x 
= 2 C_1(T,T) x + C_2(T,T) 
= 2 C_1(T,g(T,x)) x + C_2(T,g(T,x)) \label{Eqdwdxcase1} 
\end{equation}
and
\begin{align}
\frac{\partial w}{\partial T} (T,x) &= C'_0(T)  x^2. \label{Eqdwdtcase1}
\end{align} 
\item[(ii)] $x \in ( \frac{\gamma}{2 \theta C_0(t)}, \bar{X}(T,0))$. 
This is equivalent to 
$x=\bar{X}(T,S)$, $g(T,x)=S$ for some $S \in (0,T)$. We have  
$y \in ( \frac{\gamma}{2\theta C_0(v)},\bar{X}(T,0))$ and therefore $g(v,y) \in (0,T)$
for all $(v,y)$ in some neighborhood of $(T,x)$; thus
\[
w(v,y) =C_1(v,g(v,y)) y^2 + C_2(v,g(v,y)) y +C_3(v,g(v,y)).
\]
Using Lemma~\ref{LemAdvDiffg} and Lemma~\ref{LemAdvMainDiff}~(i), we obtain
\begin{equation}
\frac{\partial w}{\partial x} (T,x) = 2 C_1(T,g(T,x)) x + C_2(T,g(T,x)) \label{Eqdwdxcase2}
\end{equation} 
and
\begin{equation}
\frac{\partial w}{\partial T} (T,x) =  \frac{\partial C_1}{\partial T}(T,g(T,x)) x^2 
	+ \frac{\partial C_2}{\partial T}(T,g(T,x)) x +\frac{\partial C_3}{\partial T}(T,g(T,x)). \label{Eqdwdtcase2}
\end{equation} 
\item[(iii)] $x > \bar{X}(T,0)$. 
In this case we have $y > \bar{X}(v,0)$ and $g(v,y)=0$ for all $(v,y)$ 
in some neighborhood of
$(T,x)$ (cf.~Equation~\eqref{EqAdvDefg1} again); thus
$
w(v,y)=C_1(v,0) y^2 + C_2(v,0) y +C_3(v,0).
$
Therefore,
\begin{align} 
\frac{\partial w}{\partial x} (T,x)  &=  2 C_1(T,0) x + C_2(T,0) 
= 2 C_1(T,g(T,x)) x + C_2(T,g(T,x)), \label{Eqdwdxcase3} \\
\frac{\partial w}{\partial T} (T,x)  &=   \frac{\partial C_1}{\partial T} (T,0) x^2 
	+ \frac{\partial C_2}{\partial T} (T,0) x +\frac{\partial C_3}{\partial T}(T,0). \label{Eqdwdtcase3}
\end{align} 
\item[(iv)] $x \in  \{ \frac{\gamma}{2 \theta C_0(t)}, \bar{X}(T,0)\}$. Recall first that $w$ is continuous. Furthermore, as $g(T,\cdot)$ is continuous in $x$ with values in $[0,T]$ and $C_1$ and $C_2$ are continuous in $[0,T]$,  $\frac{\partial w}{\partial x} (T,\cdot)$ can be extended continuously to $\mathds{R}$ (cf. also Equations~\eqref{Eqdwdxcase1}, \eqref{Eqdwdxcase2} and 
\eqref{Eqdwdxcase3}). Let now $h_n\not=0$ ($n\in \mathds{N}$), $h_n \rightarrow 0$. For large enough $n$, 
the mean value theorem implies that there exists $\xi_n$ strictly between $x$ and $x-h_n$ such that
$
\frac{w(T,x)-w(T,x-h_n)}{h_n} =\frac{\partial w}{\partial x} (T,\xi_n).
$ 
As $\lim_{n\rightarrow \infty} \xi_n=x$, this implies
\begin{equation*}
\frac{\partial w}{\partial x} (T,x)= \lim\limits_{n\rightarrow \infty}
\frac{w(T,x)-w(T,x-h_n)}{h_n} 
=\lim\limits_{n\rightarrow \infty}\frac{\partial w}{\partial x} (T,\xi_n),
\end{equation*} 
i.e., $w(T,\cdot)$ is differentiable in $x$ with derivative
\begin{equation*}
\frac{\partial w}{\partial x} (T,x) = \lim\limits_{n\rightarrow \infty}\frac{\partial w}{\partial x} (T,\xi_n)
= 2 C_1(T,g(T,x)) x + C_2(T,g(T,x)).
\end{equation*}

For differentiation with respect to $T$, note that by the differential equations for $C_0(\cdot)$,
$C_1(\cdot,S)$, $C_2(\cdot,S)$ and $C_3(\cdot,S)$, we obtain for $x=\frac{\gamma}{2 \theta C_0(t)}$ that
\[
\frac{\partial C_1}{\partial T}(T,T) x^2 + \frac{\partial C_2}{\partial T}(T,T) x 
+ \frac{\partial C_3}{\partial T}(T,T) = C'_0(T) x^2.
\] 
Therefore, $\frac{\partial w}{\partial T} (\cdot,x)$ can be extended continuously to $(0,\infty)$ by Equations~\eqref{Eqdwdtcase1}, \eqref{Eqdwdtcase2} and \eqref{Eqdwdtcase3}. 
Using the mean value
theorem in a similar fashion as before, we obtain that $w(\cdot,x)$ is differentiable in $T$ 
with derivative
\begin{equation*}
\frac{\partial w}{\partial T} (T,x) =
\begin{cases}
\frac{\partial C_1}{\partial T}(T,T) x^2 + \frac{\partial C_2}{\partial T}(T,T) x 
	+ \frac{\partial C_3}{\partial T}(T,T)
= C'_0(T) x^2 &\text{if } x=\frac{\gamma}{2 \theta C_0(T)} \\ 
\frac{\partial C_1}{\partial T}(T,0) x^2 + \frac{\partial C_2}{\partial T}(T,0) x 
	+ \frac{\partial C_3}{\partial T}(T,0) &\text{if } x=\bar{X}(T,0).
\end{cases}
\end{equation*}
\end{enumerate}
\end{proof}

By Theorem~\ref{TheoremAdvContDiffw}, we obtain
first order partial derivatives
of $w$. 
We can deduce strict convexity of
$w$ (and hence $\tilde{w}$) in $x$ by computing the second order partial derivative
$
\frac{\partial^2 w}{\partial x^2}.
$
The key step towards this goal was accomplished in Lemma~\ref{LemAdvMainDiff}~(ii).

\begin{theorem} \label{CorAdvConvexx}
The candidate value function $\tilde{w}$ given via $w$ as in Equation~\eqref{EqAdvCandidateValueFct} is strictly convex in $x \in \mathds{R}$.
\end{theorem}

\begin{proof}
Let $T>0$ and $x\in \mathds{R}$. We consider the case $x\geq 0$. The case $x<0$ can be proven accordingly. 
Thus, $x=\bar{X}(T,S)$ for some $S \in (0,T)$ or $x \geq \bar{X}(T,0)$. In the first case, Theorem~\ref{TheoremAdvContDiffw}, Lemma~\ref{LemAdvDiffg} and Lemma~\ref{LemAdvMainDiff}~(ii) yield
$
\frac{\partial^2 w}{\partial x^2} (T,x) = 2 C_1(T,g(T,x))$.
For $x > \bar{X}(T,0)$, we obtain
$
\frac{\partial^2 w}{\partial x^2} (T,x)  =  2 C_1(T,0).$
Hence, $\frac{\partial^2 w}{\partial x^2} (T,x)$ can be extended
continuously to $x=\bar{X}(T,0)$ and we obtain
\begin{equation}
\frac{\partial^2 w}{\partial x^2} (t,\bar{X}(T,0))  =  2 C_1(T,0) \notag
\end{equation}
by the same argument as in the proof of Theorem~\ref{TheoremAdvContDiffw}.
Therefore (and by the respective considerations for $x<0$), the second partial derivative of $w$ with respect to $x$ exists and
$
\frac{\partial^2 w}{\partial x^2} (T,x) >0,
$
hence $w$ is strictly convex in $x$.
\end{proof}

\subsection{Classical solution of the HJB equation} \label{SubSectAdvHJB}

The fact that $w$ is differentiable and the formulae for the partial derivatives finally enable us to prove that $w$ is a classical solution of the HJB Equation~\eqref{EqHJBAdv}.

\begin{theorem} \label{CorAdvHJBMain}
The candidate value function $w$ given by Equation~\eqref{EqAdvCandidateValueFct} is a classical solution of the
HJB Equation~\eqref{EqHJBAdv} with unique minimizer 
$u^*=(\xi^*,\eta^*)$, where
\begin{align} 
\xi^*=\xi^*(T,x)&=
	\frac{2 C_1(T,g(T,x)) x+ \sgn(x) C_2(T,g(T,x))}{2 \lambda}, \label{EqAdvOptimalxi}\\
\eta^* =\eta^*(T,x)&= 
	\begin{cases}
	\sgn(x) \Big(| x |- \frac{\gamma}{2 \theta C_0(T)} \Big) &\text{if } |x| >\frac{\gamma}{2 \theta C_0(t)} \\
	0 &\text{if } |x| \leq \frac{\gamma}{2 \theta C_0(t)}.
	\end{cases} \label{EqAdvOptimaleta}
\end{align}
\end{theorem}

\begin{proof}
The theorem follows directly from the subsequent Lemma~\ref{lemAdvHJBLemma} and Theorem~\ref{TheoremAdvContDiffw} by using the Differential Equations~\eqref{DiffEqAdvC1}, \eqref{DiffEqAdvC2} and \eqref{DiffEqAdvC3} for the case 
$|x| > \frac{\gamma}{2 \theta C_0(t)}$ and the differential equation $C'_0(\cdot) = \alpha - \frac{C_0(\cdot)^2}{\lambda}$ 
for the case 
$|x| \leq \frac{\gamma}{2 \theta C_0(t)}$.
\end{proof}

\begin{lem} \label{lemAdvHJBLemma}
Let $h:  (0,\infty) \times \mathds{R} \times \mathds{R} \times \mathds{R} \rightarrow \mathds{R}$ be given by
\begin{equation} \notag
h(T,x,\xi,\eta) :=\theta \big( w(T,x-\eta)-w(t, x) \big) - \frac{\partial w}{\partial x} (t,x) \xi + f(\xi,\eta ,x).
\end{equation}
For fixed $T>0$ and $x\in \mathds{R}$, $h(T,x,\cdot,\cdot)$ attains its unique minimum at $u^*=(\xi^*,\eta^*)$ for $\xi^*$ and $\eta^*$ as in Equations~\eqref{EqAdvOptimalxi} and \eqref{EqAdvOptimaleta}, respectively.

Moreover,
\begin{equation} 
h(T,x,\xi^*,\eta^*) =
\begin{cases}
\Big( \alpha - \frac{C_1(T,g(T,x))^2}{\lambda}- \theta C_1(T,g(T,x)) \Big) x^2 \\
\quad +\Big(\gamma - C_2(T,g(T,x)) \Big( \frac{C_1(T,g(T,x))}{\lambda}+\theta \Big) \Big) |x| \\
\quad -\theta C_3(T,g(T,x)) - \frac{\gamma^2}{4 \theta C_0(T)} - \frac{C_2(T,g(T,x))^2}{4 \lambda} &\text{if }
	|x| > \frac{\gamma}{2 \theta C_0(T)} \\
\Big( \alpha - \frac{C_0(T)^2}{\lambda} \Big) x^2 &\text{if }
	|x| \leq \frac{\gamma}{2 \theta C_0(T)}. 
\end{cases}
\label{EqAdvheingesetzt}
\end{equation}
\end{lem}

\begin{proof}
Let $T>0$ and $x \in \mathds{R}$.
We have 
\begin{equation} \notag
h(T,x,\xi,\eta) =- \theta w(T,x) + \alpha x^2 - h_1(T,x,\xi) - h_2(T,x,\eta)
\end{equation}
for
\begin{equation*}
h_1(T,x,\xi) := \lambda \xi^2 - \frac{\partial w}{\partial x} (T,x) \xi =  \lambda \Big( \frac{\frac{\partial w}{\partial x} (T,x)}{2\lambda}- \xi \Big)^2 -
	\frac{\frac{\partial w}{\partial x} (T,x)^2}{4\lambda}, \quad
h_2(T,x,\eta) := \theta w(T, x- \eta) + \gamma |\eta|.
\end{equation*}
As $\lambda >0$, $h_1(T,x,\cdot)$ is strictly convex in $\xi$. Furthermore, $h_2(T,x,\cdot)$ is strictly convex in $\eta$ by 
Theorem~\ref{CorAdvConvexx}. Thus, $h(T,x,\cdot,\cdot)$ is strictly convex in $(\xi,\eta)$ and attains its unique global minimum in 
$(\xi^*,\eta^*)$ if $h_1(T,x,\cdot)$ attains its unique global minimum in $\xi^*$ and $h_2$ attains its unique global minimum in $\eta^*$.

By Theorem~\ref{TheoremAdvContDiffw}, $h_1$ attains its unique global minimum for $\xi^*$ as in Equation~\eqref{EqAdvOptimalxi}. 

Note now that $h_2 \geq 0$. For $x=0$, 
$
h_2(T,x,\eta^*)=h_2(T,0,0)=0,
$
and therefore $h_2$ attains its unique global minimum in $\eta^*$ as in Equation~\eqref{EqAdvOptimaleta}.

For $\eta \not=0$, $h_2$ is differentiable in $\eta$. For $|x| > \frac{\gamma}{2 \theta C_0(T)}$, we  have
\begin{equation*}
\frac{\partial h_2}{\partial \eta} (T,x,\eta^*)
=-\theta \Big( \frac{w}{\partial x}(T, x- \eta^*) - \sgn(x) \frac{\gamma}{\theta} \Big) =-\theta \Big( 2 C_0(T) \sgn(x) \frac{\gamma}{2 \theta C_0(T)} - \sgn(x) \frac{\gamma}{\theta} \Big)  =0 \notag
\end{equation*}
for $\eta^*=\eta^*(T,x)$ as in Equation~\eqref{EqAdvOptimaleta} (note that $\eta^*(T,x)\not=0$ 
for $|x| > \frac{\gamma}{2 \theta C_0(T)}$).
By strict convexity, we conclude that $h_2(T,x, \cdot)$ attains its unique minimum at $\eta^*$ as in 
Equation~\eqref{EqAdvOptimaleta}.

Let now $x\in (0,\frac{\gamma}{2 \theta C_0(t)})$. We have
\begin{align}
\lim\limits_{\eta \rightarrow 0-} \frac{\partial h_2}{\partial \eta} (T,x,\eta) 
	&=\lim\limits_{\eta \rightarrow 0-} (-2 \theta  C_0(t)  (x-\eta)  - \gamma)
	=- 2 \theta  C_0(T)  x  - \gamma < 0, \notag \\
\lim\limits_{\eta \rightarrow 0+} \frac{\partial h_2}{\partial \eta} (t,x,\eta) 
	&=\lim\limits_{\eta \rightarrow 0+} (- 2 \theta  C_0(T)  (x-\eta)  + \gamma)
	= 2 \theta  C_0(T)  x  + \gamma > 0. \notag
\end{align}
Therefore, $\eta^*=0$ minimizes $h_2$ uniquely by strict convexity as required.

By continuity of $\eta^*$ and $h_2$, $\eta^*(T,x)=0$
minimizes $h_2$ uniquely for $x=\frac{\gamma}{2\theta C_0(T)}$.

The case $x\in [-\frac{\gamma}{2 \theta C_0(t)},0)$ follows similarly as above.
We plug $(\xi^*,\eta^*)$ into $h$ and obtain Equation~\eqref{EqAdvheingesetzt}, finishing the proof.
\end{proof}

\section{Solution of the optimization problem} \label{SecVerificationAdv}

In the previous sections we collected the building blocks
for solving the Optimization Problem~\eqref{EqValueFctAdv}. The goal of this section
is to prove that~\eqref{EqValueFctAdv} is solved by the Markovian candidate for the optimal control $u^*_T=(\xi^*_T,\eta^*_T)$, where (cf.~Equations~\eqref{EqAdvOptimalxi} and \eqref{EqAdvOptimaleta})
\begin{align} 
\xi^*_T(t)=\xi^*(T-t,y)&=
	\frac{2 C_1(T-t,g(T-t,y)) x+ \sgn(y) C_2(T-t,g(T-t,y))}{2 \lambda}, \label{EqAdvOptimalxineu}\\
\eta^*_T(t) =\eta^*(T-t,y)&= 
	\begin{cases}
	\sgn(y) \Big(| y |- \frac{\gamma}{2 \theta C_0(T-t)} \Big) &\text{if } |x| >\frac{\gamma}{2 \theta C_0(T)} \\
	0 &\text{if } |x| \leq \frac{\gamma}{2 \theta C_0(T)}
	\end{cases} \label{EqAdvOptimaletaneu}
\end{align}
for state $y$ at time $t-$,
and that the value function is given by the candidate value function $w$.
Given the control $u^*_T$, we denote the process controlled by $u^*_T$ by
\[
X^*_T:=X^{u^*_T}.
\]

We proceed as follows. In Section~\ref{SubSecAdvAdmissible}, we obtain
bounds for $w$, $u^*_T$ and $X^*_T$. Using these bounds, we show that $u^*_T$ is 
an admissible control. We prove the verification theorem in Section~\ref{SubSecVerificationAdv}. For both tasks it is essential that the $w$ and $u^*_T$ are given in closed form. Firstly, this enables us to verify that the required moment bounds implied by Definition~\ref{DefAdmStr} hold for $u^*_T$ (cf.~Equation~\eqref{EqIntegrability}). Secondly, the bounds for $u^*_T$ turn into an upper bound for  $X^*_T$ via Gronwall's inequality; this is required in order to prove that the Terminal Constraint~\eqref{EqLiqConstraint} is satisfied. Finally, the closed form solutions are key for the verification argument: Because of the conjectured  singularity of the value function at terminal time $T$, we have to consider the limit $t \rightarrow T-$ and exchange integration and limits; the main step in this respect is accomplished in Lemma~\ref{LemLimitwforallu}.

\subsection{Admissibility of the candidate optimal control}  \label{SubSecAdvAdmissible}

Using Equation~\eqref{EqAdvC1}, it can be shown by direct computation of the respective partial derivative that $C_1(T,\cdot)$ is strictly increasing in $[0,T]$. By the series expansion of $\coth$, Equation~\eqref{Eq1dimesionalvaluefct} 
and Equation~\eqref{EqAdvC1T} we obtain the following inequalities.
For $T>0$ and $S_1,S_2 \in (0,T)$, $S_1<S_2$,
\begin{equation} \label{IneqAdvValueCoeff}
 0 < C(T) = C_1(T,0)  <  C_1(T,S_1)  <  C_1(T,S_2)  <  C_1(T,T)=C_0(T) 
\!\leq \! \frac{\lambda}{T} + \sqrt{\alpha \lambda}.
\end{equation} 

Using Taylor's theorem with integral remainder term, we
can deduce the following useful bounds for the candidate 
value function $w$ and the candidate optimal absolutely continuous control $\xi_T^*$.

\begin{prop} \label{PropAdvBounds}
Let $T>0$ and $x \not=0$. Then
\begin{align}
C_1(T,0) x^2 &< w(T,x) \leq C_0(T) x^2 \leq \Big(\frac{\lambda}{T} +\sqrt{\alpha \lambda}\Big) x^2,  \label{IneqAdvValueBound}\\
\frac{C_1(T,0)}{\lambda} |x| &< |\xi^*(T,x)|=|\xi^*_T(0)|  \leq \frac{C_0(T)}{\lambda} |x| 
	\leq \Big(\frac{1}{T}+\sqrt{\frac{\alpha}{\lambda}}\Big) |x|.
	\label{IneqAdvStrategyBound}
\end{align}
\end{prop}

\begin{proof}
Note first that 
$
w(T,0)=\frac{\partial w}{\partial x} (T,0) =0
$
by Equations~\eqref{EqAdvCandidateValueFct} 
and~\eqref{EqAdvdwdx} (recall that $g(T,0)=T$). Thus,
\begin{equation*}
w(T,x) = w(T,0) + \frac{\partial w}{\partial x} (T,0) x + \int_0^x 
	\frac{\partial^2 w}{\partial x^2} (T,y) (x-y) dy =2 \int_0^x C_1(T,g(T,y)) (x-y) dy \notag
\end{equation*}
(cf.~the proof of Theorem~\ref{CorAdvConvexx} for the second equality). Using~\eqref{IneqAdvValueCoeff}, we can directly deduce~\eqref{IneqAdvValueBound}. Furthermore,
$
\frac{\partial w}{\partial x} (T,x) = 2 \lambda \xi^*(T,x).
$
Thus,
\begin{equation*}
 2 \lambda \xi^*(T,x) =  \frac{\partial w}{\partial x} (T,0)  + \int_0^x 
	\frac{\partial^2 w}{\partial x^2} (T,y) dy =2 \int_0^x C_1(T,g(T,y)) dy,
\end{equation*}
and~\eqref{IneqAdvStrategyBound} follows from~\eqref{IneqAdvValueCoeff}.
\end{proof}

Using Gronwall's inequality, we can deduce an upper bound $X^*_T$. 

\begin{cor}  \label{PropAdvBoundTraj}
Let $T>0$, $x \in \mathds{R}$ and $t \in [0,T)$. Then
\begin{equation*}
|X_T^*(t)| \leq |x|  \exp \Big(-\int_0^t \frac{C_1(T-s,0)}{\lambda} ds\Big) 
	= |x| \exp\big(\frac{\lambda \theta}{2} t \big) 
	\frac{\sinh\big(\frac{\tilde{\theta}}{2}(T- t)\big)}
	{\sinh \big(\frac{\tilde{\theta}}{2}T \big)}.
\end{equation*}
\end{cor}

\begin{proof}
We define the process $(\tilde X (t))_{t \in [0,T)}$ as the solution of the initial value problem
\begin{equation} \notag
X'(t)=-\xi^*(T-t,X(t)), \quad
X(0)=x.
\end{equation}
By the structure of the process $\eta^*$ (cf.~Equation~\eqref{EqAdvOptimaletaneu}), we have
that for all $t \in[0,T)$,
$
|X^*_T(t)| \leq |\tilde X (t)|.
$
Therefore, by Gronwall's inequality, Proposition~\ref{PropAdvBounds} and Equation~\eqref{EqAdvC1T},
\begin{equation*}
|X^*_T(t)| \leq |\tilde X (t)| \leq |x|  \exp \Big(-\int_0^t \frac{C_1(T-s,0)}{\lambda} ds\Big) 
= |x| \exp\big(\frac{\lambda \theta}{2} t \big) 
	\frac{\sinh\big(\frac{\tilde{\theta}}{2}(T- t)\big)}
	{\sinh \big(\frac{\tilde{\theta}}{2}T \big)}
\end{equation*}
as required.
\end{proof}

The bounds derived above allow us to deduce admissibility of $u^*_T$.

\begin{prop}
Let $T>0$, $x\in \mathds{R}$ and $u^*_T=(\xi^*_T,\eta^*_T)$ be as in Equations~\eqref{EqAdvOptimalxineu} and \eqref{EqAdvOptimaletaneu}. Then
$u^* \in \mathbb{A}(T,x)$.
\end{prop}

\begin{proof}
Let $t\in [0,T)$. By Proposition~\ref{PropAdvBounds}, Corollary~\ref{PropAdvBoundTraj} and 
Equation~\eqref{EqValFuncAdvBelow}  we have
\begin{align*}
|\xi^*_T(t)| &= |\xi^*(T-t,X^*_T(t))|  \leq \frac{C_0(T-t)}{\lambda} |X_T^*(t)| 
\leq |x| \sqrt{\frac{\alpha}{\lambda}}  \exp\big(\frac{\lambda \theta}{2} t \big) 
	\frac{\sinh\big(\frac{\tilde{\theta}}{2}(T- t)\big)}
	{\sinh \big(\frac{\tilde{\theta}}{2}T \big)}
	\frac{\cosh \big(\sqrt{\frac{\alpha}{\lambda}} (T-t) \big)}
	{\sinh \big( \sqrt{\frac{\alpha}{\lambda}} (T-t) \big)} \\
& \longrightarrow |x| \sqrt{\frac{\alpha}{\lambda}} \exp\big(\frac{\lambda \theta}{2} T \big) 
	\frac{1}{\sinh \big(\frac{\tilde{\theta}}{2}T \big)}
	< \infty \quad \text{as } t \rightarrow T-.
\end{align*}
Thus,
$
\mathbb{E} [\int_0^T \xi_T^*(t)^2 dt ] < \infty.$ By Corollary~\ref{PropAdvBoundTraj}, we have $
\mathbb{E} [\int_0^T X_T^*(t)^2 dt ] < \infty$ and as $|\eta^*_T(t)| < |X_T^*(t)|$, we have
$
\mathbb{E} [\int_0^T |\eta_T^*(t)| ds ] < \infty$, i.e., $J(T,x,u_T^*)< \infty$. Finally, again by Corollary~\ref{PropAdvBoundTraj},
$
\lim_{t \rightarrow T-} X^*_T(t) = 0.
$
\end{proof}

\subsection{Verification} \label{SubSecVerificationAdv}

We are now finally able to prove that the Optimization Problem~\eqref{EqValueFctAdv}
is solved by the candidate optimal control in Equations~\eqref{EqAdvOptimalxineu} and~\eqref{EqAdvOptimaletaneu}
and that the value function is given as in Equation~\eqref{EqAdvCandidateValueFct}.

\begin{theorem} \label{TheoremAdvOptLiq}
The value function of the Optimization Problem~\eqref{EqValueFctAdv} is given by 
\begin{equation} \notag
v(T,x)=
C_1(T,g(T,x))x^2 + C_2(T,g(T,x)) |x| + C_3(T,g(t,x))
\end{equation}
for $T>0$, $x \in \mathds{R}^n$
(cf.~Equation~\eqref{EqAdvCandidateValueFct})
where $C_1$, $C_2$, $C_3$ and $g$ are given as in Section~\ref{SecPropValAdv}.
The $\mathbb{P} \otimes \boldsymbol \lambda$ - a.s. unique optimal control is given by 
$u^*_T=(\xi^*_T,\eta^*_T)$ as in~Equations~\eqref{EqAdvOptimalxineu} and~\eqref{EqAdvOptimaletaneu}, respectively.\footnote{Here, $\boldsymbol \lambda$ denotes the Lebesque measure on $[0,T]$.}
\end{theorem}

For the proof of the theorem, we require the following lemma, which relies on the fact that
the value function $v$ is an upper bound for the value function for $\gamma=0$ given in Equation~\eqref{EqValFuncAdvBelow}.

\begin{lem} \label{LemLimitwforallu}
Let $T>0$, $x\in \mathds{R}$ and $u \in \mathbb{A}(T,x)$. Then
\[
\lim_{t \rightarrow T-}\mathbb{E} \big[ w(T-t,X^u(t)) \big] =0
.
\]
\end{lem}

\begin{proof}
Note first that there exists a constant $K>0$ such that
\begin{equation} \label{IneqCostsMult}
\frac{K}{T} x^2 \leq C(T) x^2  \leq v(T,x),
\end{equation}
where $C(T)$ is as in Equation~\eqref{Eq1dimesionalvaluefct}.
A control $u \in \mathbb{A}(T,x)$ has finite costs $J(T,x,u)< \infty$ 
by Definition~\ref{DefAdmStr}. Let now $t\in (0,T)$.
Then
\begin{align*}
J(T,x,u) &= 
	\mathbb{E} \Big[\int_0^{t}\Big( 
	\lambda \xi(s)^2 + \gamma |\eta(s)|
	+\alpha X^{u}(s)^2 \Big) ds \Big]  + \mathbb{E} \Big[\int_{t}^T 
	\Big( \lambda \xi (s)^2 + \gamma |\eta (s)|
	+\alpha X^{u }(s)^2 \Big) ds \Big] \\
&\geq \mathbb{E} \Big[\int_0^{t} 
	\Big( \lambda \xi (s)^2 + \gamma |\eta (s)|
	+\alpha X^{u }(s)^2 \Big) ds + v(T-t,X^{u }(t)) \Big] \\
&\overset{\eqref{IneqCostsMult}}{\geq} \mathbb{E} \Big[\int_0^{t} 
	\Big( \lambda \xi (s)^2 + \gamma |\eta (s)|
	+\alpha X^{u }(s)^2 \Big) ds \Big]+ K
	\mathbb{E}  \Big[ \frac{X^{u }(t)^2}{T-t} \Big].
\end{align*}
Thus by the monotone convergence theorem,
\begin{align*}
J(T,x,u) &\geq
	\limsup\limits_{t \rightarrow T-} \Big( \mathbb{E} \Big[\int_0^{t} 
	\Big( \lambda \xi (s)^2 + \gamma |\eta (s)|
	+\alpha X^{u }(s)^2\Big) ds \Big]
	+ K \mathbb{E}  \Big[ \frac{X^{u }(t)^2}{T-t} \Big] \Big) \\
&= J(T,x,u) 
	+ K \limsup\limits_{t \rightarrow T-} \mathbb{E}  \Big[ \frac{X^{u }(t)^2}{T-t} \Big];
\end{align*}
hence
$
\lim_{t \rightarrow T-} \mathbb{E}  \big[ \tfrac{X^{u }(t)^2}{T-t} \big]=0.
$
The assertion now follows directly from Proposition~\ref{PropAdvBounds} (cf.~Inequality~\eqref{IneqAdvValueBound}).
\end{proof}

\begin{proof}[Proof of Theorem~\ref{TheoremAdvOptLiq}]
Let $T>0$ and $x \in \mathds{R}$. We can restrict our attention to those controls $u \in \mathbb{A}(T,x)$ such that the controlled process is almost surely monotone on $[0,T]$; in particular, the sign of the controlled process is (almost surely) not changed during the time horizon $[0,T]$ (cf.~Equation~\eqref{EqLiqConstraint}).\footnote{It can easily been seen that else there exists a less costly control such that the controlled process is almost surely monotone.} As a consequence, we have for all $k \in \mathds{N}$
\begin{equation} \label{EqRestrictAdm}
\mathbb{E} \Big[ \int_0^T |\xi(t)|^k + |\eta(t)|^k + |X^u(t)|^k dt \Big] < \infty.
\end{equation}

We now let $t \in [0,T)$ and apply It\^{o}'s formula to the function 
$w(T-\cdot,X^u(\cdot))$ (recall that $w$ is differentiable by Theorem~\ref{TheoremAdvContDiffw}):
\begin{align}
w(T,x) &= w(T-t,X^u(t)) + \int_0^{t} \Big(\frac{\partial w}{\partial T} (T-s,X^u(s))+ 
	\frac{\partial w}{\partial x}(T-s,X^u(s)) \xi(s) \Big)ds \notag \\*
& \qquad +\int_0^{t} \big( w(T-s,X^u(s-))
		-w (T-s, X^u(s-)- \eta(s) )\big) \pi(ds) \notag \\
& \leq w(T-t,X^u(t)) +\int_0^{t} f(\xi(s),\eta(s),X^u(s)) ds  \notag \\*
&\qquad -\int_0^{t} \theta \big( w(T-s,X^u(s-)) -w (T-s, X^u(s-)- \eta(s) ) \big)  ds  \notag \\*
& \qquad  + \int_0^{t} \big( w(T-s,X^u(s-))
		-w (T-s, X^u(s-)- \eta(s) ) \big) \pi(ds) \label{InEqAdvHJB} 	\\
& =  w(T-t,X^u(t)) + \int_0^{t} f(\xi(s),\eta(s),X^u(s)) ds \notag \\*
&\qquad +\int_0^{t}	\big( w(T-s,X^u(s-)) -w (T-s, X^u(s-)- \eta(s) ) \big)  M(ds),  \label{InEqAdvHJB2}
\end{align}
where $M$ is the compensated Poisson process given by $M(s):=\pi(s)-\theta s$ and Inequality~\eqref{InEqAdvHJB} follows from Theorem~\ref{CorAdvHJBMain}, with
equality if and only if
$
\boldsymbol\lambda[u=u^*_T]=0.
$
Taking expectations on both sides, we obtain 
\begin{align}
w(T-t,x) &\leq \mathbb{E} [w(T-t,X^u(t)) ] 
	+\mathbb{E} \Big[ \int_0^{t} f(\xi(s),\eta(s),X^u(s)) ds \Big] \notag \\
&\qquad  +\mathbb{E}\Big[
	\int_0^{t} w(T-s,X^u(s-)) -w (T-s, X^u(s-)-\eta(s) ) M(ds) \Big], 
	\label{IneqAdvItoIntMart}
\end{align}
with equality if and only if
$\mathbb{P} \otimes \boldsymbol \lambda [u=u^*_T] = 0.$ 
By Proposition~\ref{PropAdvBounds}, there exists a constant $K=K(t)$ such that
\begin{align}
& \mathbb{E} \Big[\int_0^t |w(T-s,X^u(s-))
	-w \big( T-s, X^u(s-) - \eta(s) \big) |^2 ds \Big]  \leq K(t)\mathbb{E}\Big[\int_0^t |X^u(s-)|^4 
	 ds \Big]  
 <\infty \notag
\end{align}
(cf.~the discussion above). As $\langle M\rangle (s) = \theta s$, this implies that the stochastic integral in Equation~\eqref{InEqAdvHJB2}
is a martingale.
Thus, taking the limit $t \rightarrow T$ in Inequality~\eqref{IneqAdvItoIntMart}, 
we obtain by Lemma~\ref{LemLimitwforallu} and the monotone convergence theorem that
\[
w(T,x) 
 \leq  \mathbb{E} [ \int_0^{T} f(\xi(s),\eta(s),X^u(s)) ds ]
=J(T,x,u),
\]
again with equality if 
$ \mathbb{P} \otimes \boldsymbol \lambda[u=u^*_T] = 0.$

For uniqueness, let $u=(\xi,\eta),\tilde{u}=(\tilde{\xi},\tilde{\eta}) \in \mathbb{A}(T,x)$
and $\mu\in (0,1)$. We define the convex combination $\bar{u}=(\bar{\xi},\bar{\eta})$:
\[
\bar{\xi}(t) =\mu \xi(t) + (1-\mu) \tilde{\xi}(t), \quad 
\bar{\eta}(t) =\mu \eta(t) + (1-\mu) \tilde{\eta}(t)
\]
for $t \in [0,T)$. Thus,
$
X^{\bar{u}}(t) = \mu X^u(t) + (1-\mu) X^{\tilde{u}}(t)
$
and $\bar{u} \in \mathbb{A}(T,x)$.
Notice that 
\begin{equation} \label{ImplicationEqual}
\mathbb{P}\otimes\boldsymbol \lambda\big[u \not= \tilde{u}\big]>0
\quad \text{implies} \quad \mathbb{P}\otimes\boldsymbol \lambda \big[\xi \not= \tilde{\xi}\big] >0
\end{equation}
as else
$
\mathbb{P}[ \lim_{t\rightarrow T-} X^u(t)
	\not=  \lim_{t\rightarrow T-} X^{\tilde{u}}(t) ]>0,
$
a contradiction to Definition~\ref{DefAdmStr}.
Hence by the convexity of $f$,
\begin{align*}
J(T,x,\bar{u}) &= \mathbb{E} \Big[ \int_0^T 
		f\big(\bar{\xi}(t),\bar{\eta}(t), X^{\bar{u}}(t) \big) dt \Big] \\
	&\leq \mathbb{E} \Big[ \int_0^T 
		\mu f\big( \xi(t), \eta(t) , X^u(t)\big) + (1-\mu) f \big(\tilde{\xi}(t),\tilde{\eta}(t), X^{\tilde{u}}(t) \big) dt \Big] 
	= \mu J(T,x,u) + (1-\mu) J(T,x, \tilde{u});
\end{align*}
by the strict convexity of
$f$ in the first argument and~\eqref{ImplicationEqual}, we have equality if and only if $u = \tilde{u}$ $\mathbb{P}\otimes\boldsymbol \lambda$ - a.s.
\end{proof}

\section{Optimal liquidation in dark pools with adverse selection} \label{SecAdvProperties}

In this section, we apply the solution of the Optimization Problem~\eqref{EqValueFctAdv} to a cost minimization problem arising in the context of optimal portfolio liquidation if a large investor has access both to a classical exchange and to a dark pool with adverse selection (cf.~the discussion in the introduction). In Section~\ref{ModelAdverse}, we describe the market model and show how Theorem~\ref{TheoremAdvOptLiq} applies to the model. In Section~\ref{PropAdv}, we discuss the properties of the solution of the Optimization Problem~\eqref{EqValueFctAdv} with regard to the application.

\subsection{Model description} \label{ModelAdverse}

Our model for trading and price formation at the classical exchange is a linear price impact model. Trade execution can be enforced by selling aggressively, which however results in quadratic execution costs due to a stronger price impact. We model order execution in the dark pool by a Poisson process. Orders submitted to the dark pool are executed at the jump times of Poisson processes; however, these orders result in adverse selection costs. The split of orders between dark pool and exchange is thus driven by the trade-off between execution uncertainty, price impact costs and adverse selection costs. Our model is a generalization of the single-asset version of the model of~\cite{Kratz2012} who neglect adverse selection.

In the following, we first specify the transaction prices and the trade execution in the primary venue and the dark pool (Section~\ref{SubSecModelPrimaryEx}); in this context, we also specify how the two venues are connected via adverse selection. Subsequently, we define admissible liquidation strategies (Section~\ref{SubSecModelStrategies}). Finally, we specify the proceeds of liquidating the position and show how Theorem~\ref{TheoremAdvOptLiq} applies to the market model in Section~\ref{SubSecModelCosts}.

\subsubsection{Transaction prices and trade execution} \label{SubSecModelPrimaryEx}

For a fixed time interval $[0,T]$, we consider the stochastic basis $(\Omega, \mathcal{F}, \mathbb{P}, \mathbb{F}=(\mathcal{F}_t)_{t \in[0,T]})$\footnote{The filtration is generated by the involved random processes and is specified after Assumption~\ref{AssPrimary}.}. 

In absence of transactions of the investor, the fundamental asset price at the primary exchange is given by a stochastic process $\tilde{P}$. We assume that $\tilde{P}$ is connected to the liquidity in the dark pool through adverse selection. We model this liquidity by Poisson processes \[
\pi_1 \quad \text{respectively} \quad \pi_2 \quad \text{ 
with the same intensity} \quad
\theta > 0:
\] 
sell respectively buy orders (i.e., positive respectively negative orders) in the dark pool are executed \emph{fully} at the jump times of $\pi_1$ respectively $\pi_2$; else, the orders are not executed at all. This simplification allows a thorough mathematical analysis of the model. On the other hand, the resulting model captures the stylized facts of dark pools outlined in the introduction. We assume the following relation between the fundamental price at the exchange, $\tilde{P}$, and the dark pool liquidity, $(\pi_1,\pi_2)$.

\begin{assumption} \label{AssPrimary}
\begin{enumerate}
\item[(i)]
Let $\bar{P}$ be a square-integrable positive c\`{a}dl\`{a}g martingale such that
the variance of $\bar{P}$ is constant in time, i.e., for all $t \in[0,T]$,
$
\var ( \bar{P}(t) ) =t \sigma^2
$
for $\sigma^2\geq 0$.
Then, $\tilde{P}$ is given by
\[
\tilde{P}(t) = \bar{P}(t) +  \Gamma \big( \Delta \pi_1(t) - \Delta \pi_2(t) \big),
\]
where $\Gamma > 0$.
\item[(ii)]
$\pi_1$, $\pi_2$ and $\bar{P}$ are independent.
\end{enumerate}
\end{assumption}

We are now ready to specify the filtration $(\mathcal{F}_t)_t$ as the completion of $\big(\sigma\big( \bar{P}(s), \pi_1(s),\pi_2(s) | 0 \leq s \leq t \big)\big)_t.$

Assumption~\ref{AssPrimary} captures the nature of adverse selection in a rather direct way: whenever there is liquidity in the dark pool (i.e., the respective Poisson process $\pi_i$ jumps), there is a favorable jump of the fundamental asset price $\tilde{P}$ of size $\Gamma$. As the intensities of $\pi_1$ and $\pi_2$ are equal, the resulting process is still a c\`{a}dl\`{a}g martingale. Positivity of $\tilde{P}$ cannot be ensured any more. Mathematically this is irrelevant as we shall see below. For the application we have in mind, time horizons are in general short and for appropriate parameters the probability that the price becomes negative is negligible. The idea to incorporate adverse selection by constructing the price process as in~Assumption~\ref{AssPrimary} is due to~\cite{Naujokat2011}. 

Once the trader becomes active at the primary exchange, she influences the market price $P$. We assume that the trader at the primary exchange can only execute trades continuously, i.e., that her trading activity on the primary exchange is absolutely continuous and can hence be described by her trading intensity $\xi(t)$ with $t \in [0,T)$. This trading in the traditional exchange generates price impact, which we assume to be temporary and linear in the trading rate $\xi(t)$. Given a strategy $(\xi(t))_{t\in [0,T)}$, the transaction price at time $t \in [0,T]$ is given by 
\[
P(t) = \tilde{P}(t) - \lambda \xi(t), \quad \text{where} \quad
\lambda >0.
\]
By assuming linear price impact for the primary venue, we follow~\cite{Almgren2001}. This choice yields a tractable model which nevertheless captures price impact effects. Linear price impact models have become the basis of several theoretical studies, e.g.,  \cite{Almgren2007}, \cite{Carlin2007}, \cite{Schoneborn2007} and~\cite{Rogers2010}. 

We allow for continuous updating of the orders $\eta(t)$ in the dark pool at any time $t \in [0,T]$.
While the dark pool has no impact on prices at the primary venue, it is less clear to which extent the price
impact of the primary venue $\lambda \xi(t)$ is reflected in the trade price of the dark pool. If for example the price impact is realized predominantly in the form of a widening spread, then the impact on dark pools that monitor the
mid quote can be much smaller than $\lambda \xi(t)$. We make the simplifying assumption that trades in the dark pool are not influenced by the price impact at all, i.e., that they are executed at the fundamental
price $\tilde{P}$ (more precisely, $\tilde{P}_{-}$; note that this distinction is irrelevant for the absolutely continuous strategy $\xi$ at the exchange).
If alternatively the transaction price in the dark pool is the price $P$ at the primary exchange including the trader's price impact, market manipulation strategies can become profitable unless the parameters are chosen with great care, as has been shown by~\cite{Kratz2010} for the discrete-time case. For a detailed discussion see also~\cite{Kloeck2011} who analyze the circumstances that can lead to price manipulation in dark pools.

\subsubsection{Admissible liquidation strategies} \label{SubSecModelStrategies}

We investigate an investor who has to liquidate an asset position $x \in \mathds{R}$ within a finite trading horizon $[0,T]$. Given a trading strategy $u=(\xi,\eta)$, the asset holdings of the investor at time $t\in [0,T)$ are given by (cf.~the analogy to Equation~\eqref{EqCSDE})

\begin{equation} \label{EqAssetHoldings}
X^u(t):=  x- \int_0^t \xi(s) ds -  \int^t_0\mathds{1}_{\{\eta(s)>0\}} \eta(s) d\pi_1(s) - \int^t_0\mathds{1}_{\{\eta(s)<0\}} \eta(s) d\pi_2(s)
\end{equation}

Similarly as in Section~\ref{ChaptPortfolioSecModel}, a trading strategy $u=(\xi,\eta)$ is \emph{admissible} ($\in \mathbb{A}(T,x)$) if $\xi$ is progressively measurable, $\eta$ is predictable, the liquidation costs $J(T,x,u)$ (defined by Equation~\eqref{EqCostsDP} below) are finite and the position is liquidated by time $T$:
\begin{equation} \label{EqLiqConstraintDP}
\lim\limits_{t \rightarrow T-} X^u(t)= 0 \quad \text{a.s.}
\end{equation}
(cf.~Definition~\ref{DefAdmStr}).

\subsubsection{Trading proceeds/liquidation costs} \label{SubSecModelCosts}

The proceeds of selling\footnote{In the following, we say \emph{selling} for both positive and negative orders.} the portfolio $x \in \mathds{R}$ during $[0,T]$ according to the strategy
$(u(t))_t=(\xi(t),\eta(t))_t \in \mathbb{A}(T,x)$ are given by
\[
\phi(T,x,u) := \int_0^T \xi(t)(\tilde{P}(t-) - \lambda \xi(t))dt 
+ \int_0^T \mathds{1}_{\{\eta(t)>0\}} \eta(t) \tilde{P}(t-)  d\pi_1(t) 
+ \int_0^T \mathds{1}_{\{\eta(t)<0\}} \eta(t) \tilde{P}(t-)  d\pi_2(t).
\]
The first term in the above equation represents the proceeds of selling at the primary exchange at a price of $P(t-) = \tilde{P}(t-) - \lambda \xi(t)$, while the second and the third term accounts for the proceeds of selling in the dark pool at the unaffected price $\tilde{P}(t-)$. Applying integration by parts and using the fact that $X^u$ satisfies Equation~\eqref{EqAssetHoldings} (cf.~also the Liquidation Constraint~\eqref{EqLiqConstraintDP}), we obtain
\begin{equation} \label{EqAdvStochInt1}
\phi(T,x,u) = -\int_0^T \lambda \xi(t)^2 dt + x^\top \tilde{P}(0) + \int_0^T X^u(t-) d\tilde{P}(t)
+\sum\limits_{0\leq t \leq T} \Delta X^u(t) \Delta \tilde{P}(t) .
\end{equation}
For the compensated Poisson processes $M_i(t):=\pi_i(t) - \theta t$, Assumption~\ref{AssPrimary} implies the following form for the last summand.
\begin{align} 
\sum\limits_{0\leq t \leq T} \Delta X^u(t) \Delta \tilde{P}(t) 
	&=- \Gamma \Big(\int_0^T \mathds{1}_{\{\eta(t)>0\}} \eta(t)  d\pi_1(t) 
		+ \int_0^T \mathds{1}_{\{\eta(t)<0\}} \eta(t) d\pi_2(t) \Big) \notag \\
	&=- \theta \Gamma  \int_0^T | \eta(t) |   dt
		- \Gamma \Big(\int_0^T \mathds{1}_{\{\eta(t)>0\}} \eta(t)  dM_1(t) 
		+ \int_0^T \mathds{1}_{\{\eta(t)<0\}} \eta(t) dM_2(t) \Big). \label{EqAdvStochInt2}
\end{align}
For appropriate integrability assumptions on $u$, the stochastic integrals in Equations~\eqref{EqAdvStochInt1} and~\eqref{EqAdvStochInt2} are true martingales and we obtain
\[
\mathds{E} \big[\phi(T,x,u)\big] = x^\top \tilde{P}(0) -\mathds{E} \Big[
\int_0^T \lambda \xi(t)^2 dt +  \int_0^T  \theta \Gamma  | \eta(t) |   dt \Big].
\]
Instead of maximizing expected proceeds, we can hence equivalently minimize expected price impact and adverse selection costs. We assume that the trader is not only interested in expected liquidation proceeds, but in addition also wants to minimize risk during liquidation. We incorporate both aspects in the following cost functional (cf.~Equation~\eqref{Costfunctional}):
\begin{align} 
J(T,x,u)&:=x^\top \tilde{P}(0) - \mathds{E} \big[\phi(T,x,u)\big] \!+\! \mathds{E}\Big[ \tilde \alpha\!\int_0^T\!\!\! \sigma^2 X^u(t)^2 dt \Big] =  \mathbb{E} \Big[ \int_0^T \!\! \!\big( \lambda \xi(s)^2 + \theta \Gamma |\eta(t)| + \tilde \alpha \sigma^2 X^u(t)^2 \big) dt \Big]. \label{EqCostsDP}
\end{align}
The first two terms in the cost functional capture the expected liquidation shortfall, while the last term is an additive penalty function $\tilde \alpha \int_0^T \sigma^2 X^u(t)^2 dt$ which reflects the market risk of the asset position;\footnote{Here, $\tilde \alpha$ is the personal risk-aversion parameter of the investor.} it penalizes slow liquidation and poorly balanced portfolios. It does not incorporate liquidity risk, however \cite{Kratz2012} argue that for realistic parameters, market risk outweighs liquidity risk which we therefore neglect. For deterministic liquidation strategies without dark pools, the risk term reflects the variance of the liquidation costs (see~\cite{Almgren2001}). In this case, minimizing a mean-variance functional of the liquidation costs over all deterministic strategies is equivalent to maximizing the expected utility of the proceeds of an investor with CARA preferences over all strategies (see~\cite{Schied2010}). 

The goal of the investor is to minimize her liquidation costs $J(T,x,u)$. For $\gamma=\theta\Gamma$ and $\alpha= \tilde \alpha \sigma^2$, we can apply Theorem~\ref{TheoremAdvOptLiq}.\footnote{Recall that $\pi_1$ and $\pi_2$ are independent by Assumption~\ref{AssPrimary}~(ii). Mathematically, we can therefore replace them by one Poisson process $\pi$ in Equation~\eqref{EqAssetHoldings}.} The minimal liquidation costs are therefore given by $v(T,x)$ as in Equation~\eqref{EqAdvCandidateValueFct}, while the optimal liquidation strategy $u^*_T$ is given by Equations~\eqref{EqAdvOptimalxineu} and~\eqref{EqAdvOptimaletaneu}.

\subsection{Properties} \label{PropAdv}

We conclude by discussing the properties of the value function and the optimal strategy.
In Section~\ref{SubSecAdvPropertiesRisk} we discuss the dependence of the
optimal strategy and the value function on the adverse selection parameter $\Gamma=\frac{\gamma}{\theta}$.
In Section~\ref{SubSecAdvPropertiesNoRisk} we analyze the case of a risk-neutral investor ($\alpha = \tilde\alpha \sigma^2=0$),
which we had excluded in Sections~\ref{SecPropValAdv} and~\ref{SecVerificationAdv}.

\subsubsection{Risk-averse investors: $\pmb{\alpha=\tilde\alpha \sigma^2>0}$} \label{SubSecAdvPropertiesRisk}

Theorem~\ref{TheoremAdvOptLiq} confirms the structure of the optimal strategy and 
the value function of the Optimization Problem~\eqref{EqValueFctAdv} we had expected in the heuristics
of Section~\ref{SecAdvHeuristic} (cf.~also Figure~\ref{FigHeu}).

Both for large initial asset positions $x$ ($|x| \geq \bar{X}(T,0)$) and for small initial asset positions 
($|x| \leq \bar{X}(T,T)=\beta(T) =\frac{\Gamma}{ 2 C_0(T)}$), the value function is a quadratic polynomial.
In between, it is an ``interpolation'' of these polynomials. 

The value function and the optimal strategy for 
$|x| \leq \frac{\Gamma}{2 C_0(T)}$ are the same as the ones without dark pool and
without adverse selection (i.e., the optimal order in the dark pool is zero). 

For larger asset positions, the absolute value of order in the dark pool is 
greater than zero; after the execution of the optimal dark pool order at time $\tau$,
the asset position is $\frac{\Gamma}{ 2 C_0(T-\tau)}$. The optimal trading trajectory
until dark pool execution
for an initial asset positions $|x| = \bar{X}(T,S) \in (\bar{X}(T,T),\bar{X}(T,0))$
is given by the function $\bar{X}(\cdot,S)$ in $[0,T-S]$.

It is interesting to examine the dependence of the value function and the optimal strategy on 
the adverse selection parameter
$\Gamma$. Intuitively, costs should be higher for large adverse selection and therefore
the value function should be increasing in $\Gamma$. Furthermore, the
dark pool is less attractive for large adverse selection and therefore
optimal dark pool orders should be decreasing in $\Gamma$, whereas trading intensity
in the primary venue should be increasing in $\Gamma$ (as impact costs are relatively lower with 
respect to adverse selection costs). The following proposition confirms these intuitions. Additionally, (i) of the proposition shows that adverse selection is negligible for large initial asset positions $|x| >> \bar{X}(T,0)$ as in this case $v(T,x) \approx C_1(T,0) x^2$ which is independent of $\Gamma$.
To stress the dependence on $\Gamma$, we will add it as an argument
for the remainder of the section and write
\[
\bar{v}(T,x,\Gamma):=\bar{v}(T,x), \quad u^*_T(t,\Gamma):=u^*_T(t),\quad  \cdots 
\]

\begin{prop} \label{PropProp}
Let $x \in \mathds{R}$, $T>0$ and $S \in[0,T]$. Then
\begin{enumerate}
\item[(i)]
$C_1(T,S,\cdot)$ is constant, $C_2(T,S,\cdot)$ is strictly increasing and
$C_3(T,S,\cdot)$ is strictly decreasing.
\item[(ii)]
$\bar{X}(T,S,\cdot)$ is strictly increasing.
\item[(iii)]
$\bar{v}(T,x,\cdot)$ is strictly increasing on the interval $(0,2|x| C_0(T))$
and constant for $\Gamma > 2|x| C_0(T)$.
\item[(iv)]
$|\xi_T^*(0,\cdot)|$ is strictly increasing and
$|\eta^*_T(0,\cdot)|$ is strictly decreasing on the interval $(0 , 2|x| C_0(T))$
and constant for $\Gamma > 2|x| C_0(T)$.
\end{enumerate}
\end{prop}

\begin{proof}
\begin{enumerate}
\item[(i)]
The first assertion follows directly from the Initial Value Problem for $C_1$, \eqref{DiffEqAdvC1}. The
second and the third assertion can be deduced from that by Equations~\eqref{EqAdvC2} 
and~\eqref{EqAdvC3}, respectively.
\item[(ii)]
The assertion follows directly from Equation~\eqref{EqAdvXbar}.
\item[(iii)]
Monotonicity of the value function follows directly from the form 
of the cost functional (cf.~Equation~\eqref{Costfunctional} as long as $|\eta_T^*(0,\Gamma)|>0$,
i.e., $ \Gamma \in (0, 2|x| C_0(T))$ (cf.~Equation~\eqref{EqAdvOptimaletaneu}).
\item[(iv)]
For the first assertion, let without loss of generality $x>0$ and
$\Gamma < \tilde{\Gamma} < 2 x C_0(T)$. By~(ii) and the fact that $\bar{X}$ is strictly decreasing in $S$ (cf.~also~Equation~\eqref{EqAdvDefg2}), 
$
g(T,y,\Gamma) \leq  g(T,y,\tilde \Gamma)
$
for all $y \in [0,x]$ with strict inequality for
$
\frac{\tilde{\Gamma}}{2 C_0(T)} < y < \bar{X}(T,0,\Gamma).
$
Therefore (cf.~the proof of Proposition~\ref{PropAdvBounds}),
\[
\xi^*_T(0,\Gamma)  =\frac{1}{\lambda} \int_0^x C_1(T,g(T,y,\Gamma),\Gamma) dy
	< \frac{1}{\lambda} \int\limits_0^x C_1(T,g(T,y,\tilde{\Gamma}),\tilde{\Gamma}) dy =\xi^*_T(0,\tilde{\Gamma}) 
\] 
by~(i) and the fact that $C_1$ is strictly increasing in $S$ (cf.~\eqref{IneqAdvValueCoeff}).

The assertion that $|\eta^*_T(0,\cdot)|$ is strictly decreasing follows directly from Equation~\eqref{EqAdvOptimaletaneu}. 
\end{enumerate}
\end{proof}

In Figure~\ref{FigProp}, we compare the optimal strategy with adverse selection (left picture) and without adverse selection (right picture) for a numerical example. In both pictures, thick solid lines denote realized trading trajectories. For those scenarios where the dark pool order is executed at time $\tau < T$, the dotted lines refer to the case where the dark pool order is never executed. The dashed line in the left picture corresponds to the boundary $\beta$ and the thin solid line in the right picture corresponds to optimal liquidation without dark pools. For a large asset position, the boundary $\beta$ is never crossed unless the dark pool is executed (upper solid line respectively dotted lines in the left picture). At all times, the dark pool order is such that the after execution, the position is exactly on the boundary. Afterwards, it stays (strictly) below the boundary for the entire trading horizon. In comparison to the optimal strategy without adverse selection (thick solid line in the right picture), the trading intensity is higher (cf.~Proposition~\ref{PropProp}~(iv)); it is however lower than the trading intensity without dark pools. Dark pools slow down trading in the primary venue, as the trader aims to reduce her impact costs and hopes to trade cheaper in the dark pool (see~\cite{Kratz2012}). This effect is decreased significantly by the introduction of adverse selection costs. Furthermore, the optimal order in the dark pool is smaller than without adverse selection, where always the \emph{full} remainder of the position is placed in the dark pool. The trader does not want to to liquidate the full position as she expects a favorable impending price move which she does not want to miss out completely. For the smaller asset position, the boundary $\beta$ is crossed at time $T-S$, i.e., the optimal position until $T-S$ is $\bar{X}(T,S)$ (unless the dark pool order is executed before). After the position crosses the boundary it stays (strictly) below it until the end of the trading horizon.

\begin{figure}
\centering
\begin{tabular}{ll}
\begin{tabular}{l}
\begin{overpic}[height=3.7cm, width=6cm]{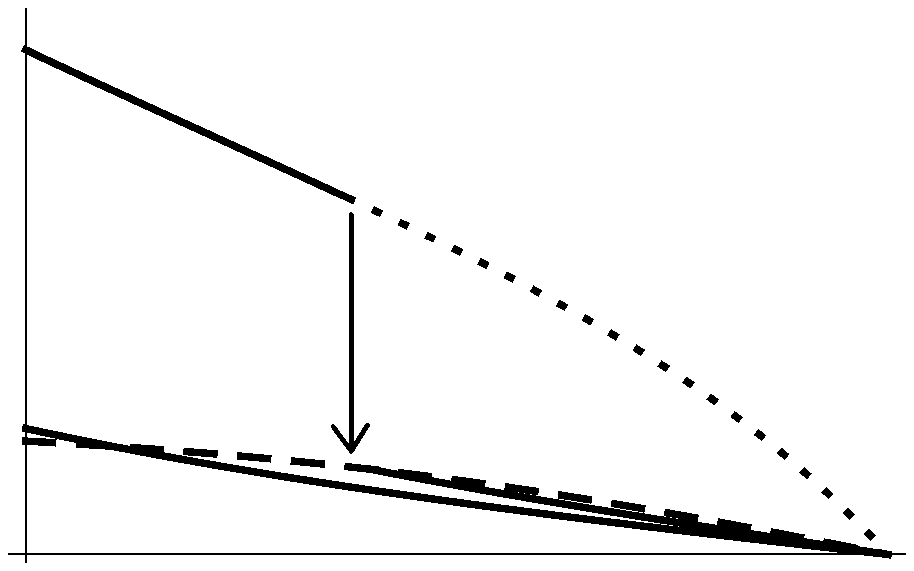}
\put(93,-3.5){\scriptsize{$T$}}
\put(5,-2.5){\scriptsize{$T-S$}}
\put(38,-2.5){\scriptsize{$\tau$}}
\put(0,62){\scriptsize{Size of asset position}}
\end{overpic}
\end{tabular} &
\hspace{1cm}
\begin{tabular}{l}
\begin{overpic}[height=3.7cm, width=6cm]{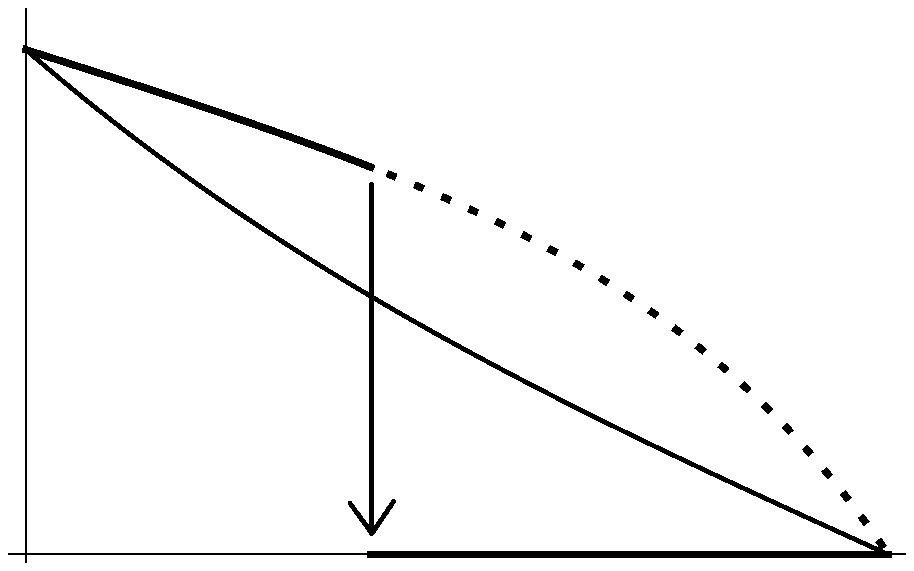}
\put(93,-3.5){\scriptsize{$T$}}
\put(39,-2.5){\scriptsize{$\tau$}}
\put(0,62){\scriptsize{Size of asset position}}
\end{overpic}
\end{tabular}
\end{tabular}\caption{Optimal trading strategies with adverse selection (left picture) and without adverse selection (right picture). In both pictures, thick solid lines denote realized trading trajectories. For those scenarios where the dark pool order is executed at time $\tau=0.4 < T=1$, the dotted lines refer to the case where the dark pool order is never executed. The dashed line in the left picture corresponds to the boundary $\beta$ and the thin solid line in the right picture corresponds to optimal liquidation without dark pools. The larger initial asset position is $x=1.2$; the smaller one is $x=0.3$ (which results in $T-S=0.1$). Furthermore, $\lambda=2.5$, $\tilde \alpha=4$ $\sigma=1$, $\theta=3$  $\Gamma=2$. 
} \label{FigProp}
\end{figure}

\subsubsection{Risk-neutral investors: $\pmb{\alpha=\tilde\alpha \sigma^2=0}$} \label{SubSecAdvPropertiesNoRisk}

In Sections~\ref{SecPropValAdv} and~\ref{SecVerificationAdv} we excluded the
case $\alpha =0$. The main reason was that in some formulae the 
term $\alpha$ appears in the denominator (cf., e.g., Equations~\eqref{EqAdvXbar},
and~\eqref{EqAdvXbarT}) and it would thus have complicated the exposition of the results. 
However, the case 
$
\alpha=\tilde \alpha \sigma^2=0
$
can be treated in a similar way as the case $\alpha=\tilde \alpha \sigma^2 > 0$. 

We define $C_1$, $C_2$, $C_3$ and $\bar{X}$ as before by the 
Initial Value Problems~\eqref{DiffEqAdvC1}~-~\eqref{DiffEqAdvXbar}.
Note that only the differential equation for $C_1$ depends on $\alpha$. Therefore
$C_2$, $C_3$ and $\bar{X}$ only depend on $\alpha$ through $C_1$. We can
compute solutions of the initial value problems directly or by taking the limits
for $\alpha \rightarrow 0$
in Equations~\eqref{EqAdvC1}~-~\eqref{EqAdvXbar}.

Note first that $\alpha=0$ implies
\[
\tilde{\theta} = \theta, \quad C_0(T) = \frac{\lambda}{T}.
\]
Therefore,
\begin{equation} \notag
\bar{X}(T,T)=\beta(T) =\frac{\Gamma}{2 C_0(T)}=\frac{\Gamma}{2 \lambda} T.
\end{equation}
We can compute the partial derivative of $\bar X$ with respect to $S$ and obtain
$
\frac{\partial \bar{X}}{\partial S} (T,S)=0;
$
therefore
\begin{equation} \label{EqAdvNoRiskXs}
\bar{X} (T,S) =\bar{X}(T,T)=\frac{\Gamma}{2 \lambda} T
\end{equation}
for $T>0$ and $S\in [0,T]$. Hence, we expect only two different trading regions. The dark pool
is only used for $|x| > \frac{\Gamma}{2 \lambda} T$, 
and we expect the value function to be given by
\begin{equation}
w(T,x) = \begin{cases} \label{EqNoRiskAdv}
C_0(T) x^2=  \frac{\lambda}{T} x^2 & \text{if } |x| \leq \frac{\Gamma}{2 \lambda} T \\
C_1(T,0) x^2 +C_2(T,0) |x| + C_3(T,0) & \text{if } |x| > \frac{\Gamma}{2 \lambda} T.
\end{cases}
\end{equation}
As in Lemma~\ref{LemAdvMainDiff} we obtain that
\[
\frac{\partial C_1}{\partial S} (T,S) \bar{X}(T,S)^2 +
\frac{\partial C_2}{\partial s} (T,S) \bar{X}(T,S) +
\frac{\partial C_3}{\partial s} (T,S) =0,
\]
\[
2 \frac{\partial C_1}{\partial S} (T,S) \bar{X}(T,S) +
\frac{\partial C_2}{\partial S} (T,S) =0.
\]
Thus, by Equation~\eqref{EqAdvNoRiskXs},
\begin{equation*}
\frac{\partial}{\partial s} \big( C_1(T,S) \bar{X}(T,S)^2 +
C_2(T,S) \bar{X}(T,S) + C_3(T,S) \big) =0,
\end{equation*}
\begin{equation*}
\frac{\partial }{\partial s} \big( 2 C_1(T,S) \bar{X}(T,S) +
C_2(T,S) \big) =0,
\end{equation*}
in particular
\begin{equation} \notag
C_1(T,0)\Big(\frac{\Gamma}{2 \lambda} T\Big)^2 +
C_2(T,0)\frac{\Gamma}{2 \lambda} T + C_3(T,0)  =C_0(T) \Big(\frac{\Gamma}{2 \lambda} T\Big)^2,
\end{equation}
\begin{equation} \notag
2 C_1(T,0) \frac{\Gamma}{2 \lambda} T + C_2(T,0)  = 2 C_0(T) \frac{\Gamma}{2 \lambda} T.
\end{equation}
We can deduce that 
$w(T,\cdot)$ and$\frac{\partial w}{\partial x} (T,\cdot)$
are continuous. However (cf.~the proof of Theorem~\ref{CorAdvConvexx})
\[
\frac{\partial^2 w}{\partial x^2}(T,x) =
\begin{cases}
2 C_0(T)  & \text{if } |x| < \frac{\Gamma}{2 \lambda} T \\
2 C_1(T,0) < 2 C_0(T)  & \text{if } |x| > \frac{\Gamma}{2 \lambda} T.
\end{cases}
\]
Defining the candidate optimal strategy as before by
\begin{align} 
\xi_T^*(t)=\xi^*(T-t,y)&:= \frac{1}{2 \lambda} \frac{\partial w}{\partial x} (T-t,y) =
\begin{cases} \label{NoRiskAdv1}
	\frac{2 C_1(T-t,0) y+ \sgn(y) C_2(T-t,0)}{2 \lambda}& \text{if } |y| > \frac{\Gamma}{2 \lambda} (T-t) \\
	\frac{y}{T-t}& \text{if } |y| \leq \frac{\Gamma}{2 \lambda} (T-t),
\end{cases} \\
\eta^*_T(t) =\eta^*(T-t,y)&:= 
	\begin{cases} \label{NoRiskAdv2}
	\sgn(y) \Big(| y |- \frac{\Gamma}{2 \lambda} (T-t) \Big) &\text{if } |y| >\frac{\Gamma}{2 \lambda} (T-t) \\
	0 &\text{if } |y| \leq \frac{\Gamma}{2 \lambda} (T-t),
	\end{cases} 
\end{align}
where $y=X_T^*(t-)$ is the position at time $t$,
we obtain that $\xi^*(T-t, \cdot)$ is continuous but \emph{not} differentiable at
\[ 
|y| =\frac{\Gamma}{2 \lambda} (T-t).
\]
All steps of the proof of Theorem~\ref{TheoremAdvOptLiq} can be replicated for 
the case $\alpha=0$ in a straightforward manner. The solution of the Optimization 
Problem~\eqref{EqValueFctAdv} is given by Equations~\eqref{NoRiskAdv1}
and~\eqref{NoRiskAdv2}. The value function is given by $w$ as in Equation~\eqref{EqNoRiskAdv}.

\begin{figure}
\centering
\begin{tabular}{lll}
\begin{tabular}{l}
\begin{overpic}[height=2.7cm, width=4.5cm]{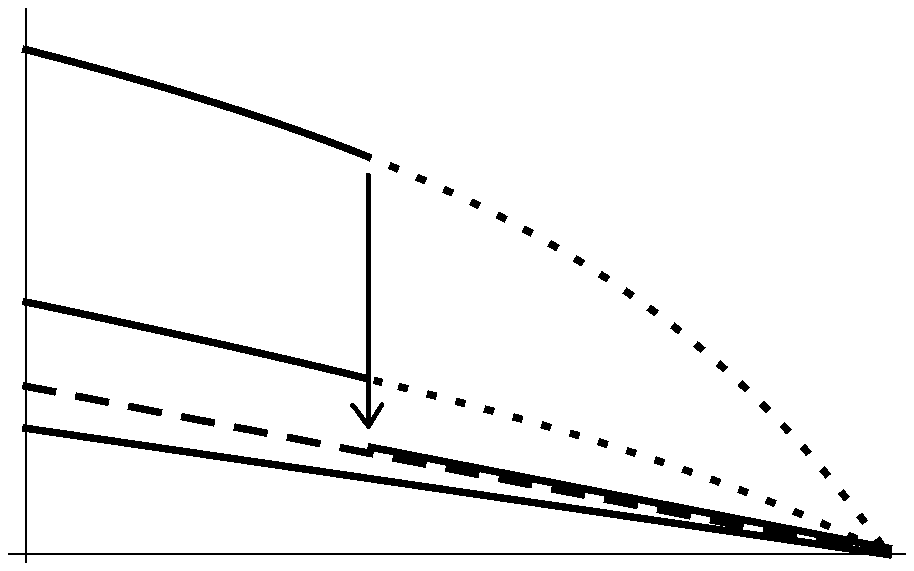}
\put(93,-3.5){\scriptsize{$T$}}
\put(38,-2.5){\scriptsize{$\tau$}}
\put(0,62){\scriptsize{Size of asset position}}
\end{overpic}
\end{tabular} &
\begin{tabular}{l}
\begin{overpic}[height=2.7cm, width=4.5cm]{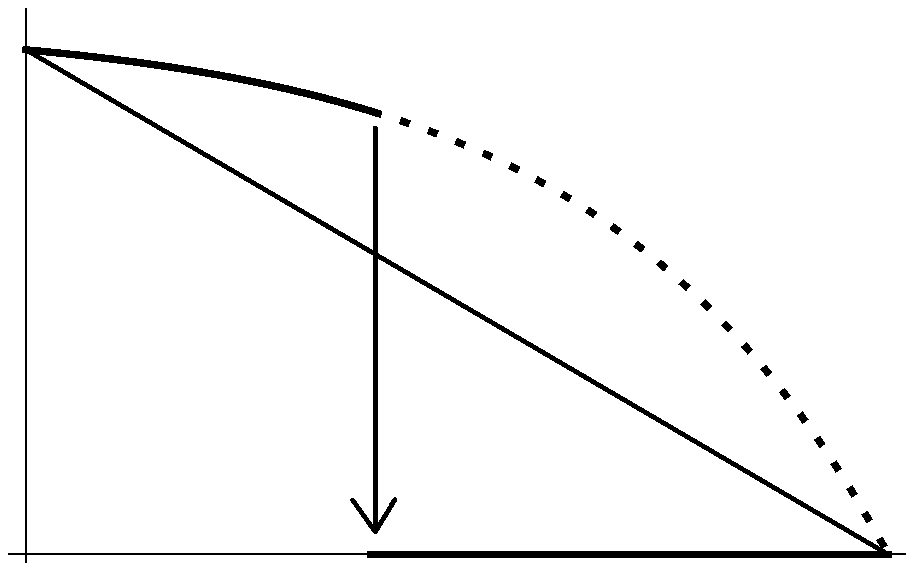}
\put(93,-3.5){\scriptsize{$T$}}
\put(39,-2.5){\scriptsize{$\tau$}}
\put(0,62){\scriptsize{Size of asset position}}
\end{overpic}
\end{tabular} &
\begin{tabular}{l}
\begin{overpic}[height=2.7cm, width=4.5cm]{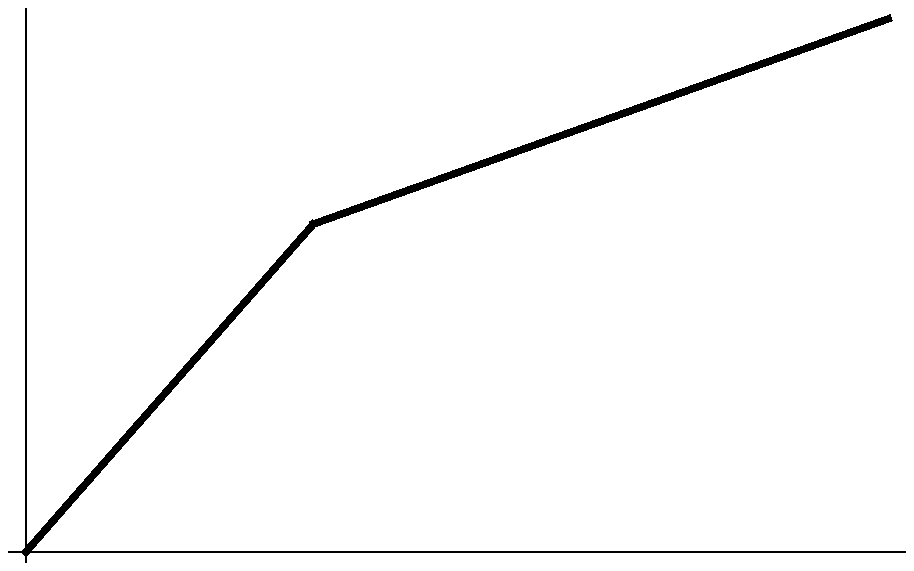}
\put(99,0.5){\scriptsize{$x$}}
\put(28,-4){\scriptsize{$\frac{\Gamma T}{2 \lambda}$}}
\put(0,62){\scriptsize{Trading intensity $\xi^*_T(0)=\xi^*(T,x)$}}
\end{overpic}
\end{tabular}
\end{tabular}\caption{The left and the middle picture illustrate the same liquidation scenarios as Figure~\ref{FigProp} for risk-neutral investors. In addition, an intermediate initial asset position is displayed ($x=0.6$). The right picture illustrates the fact that $\xi^*(T,\cdot)$ is not differentiable
at $x=\frac{\Gamma T}{2 \lambda}$. All parameters are the same as in Figure~\ref{FigProp}.}
\label{FigAdvNoRisk}
\end{figure}

We illustrate the structure of the optimal strategy in the left and the middle picture of Figure~\ref{FigAdvNoRisk}. All lines have the same meaning as in Figure~\ref{FigProp}; we only add a realized trajectory for an intermediate initial asset position which is only slightly larger than the boundary $\beta(T)=\frac{\Gamma}{2 \lambda} T$.  We observe that the optimal trading trajectory $X^*_T(t)$
never crosses the boundary $\frac{\Gamma}{2 \lambda} (T-t)$,
provided there is no dark pool execution. 
Let us briefly comment on this structure. For risk-neutral optimal liquidation
without dark pool, the optimal trading intensity is constant
(see, e.g., \cite{Kratz2012}; cf.~also the thin solid line in the middle picture which refers to risk-neutral liquidation without dark pool). The boundary itself is linear,
and it is the trading trajectory of the optimal strategy (without dark pool) with initial
position \emph{on} the boundary.
If the initial position is above the boundary, the usage of the dark pool \emph{slows down} the optimal trading intensity, i.e.,
if the dark pool is not
executed,
the trading trajectory is concave and thus never crosses the boundary. 
If the dark pool order is executed (in the displayed scenario at time $\tau$), the resulting position is again on the boundary. However, it does \emph{not cross} it but stays on it until the end of the trading horizon. Finally, the boundary $\beta$ is decreasing in $\alpha$ (cf., e.g., Equation~\eqref{EqCandidateb}). This results in the small position in the left picture to be below the boundary from the beginning. Again, it is optimal to liquidate linearly in this case.

The right picture of Figure~\ref{FigAdvNoRisk} illustrates
the fact outlined above that the optimal trading intensity $\xi^*(T-t,\cdot)$ is not differentiable
at $x=\frac{\Gamma T}{2 \lambda}$.

\appendix



\bibliographystyle{abbrvnat}
\bibliography{bibliography}

\end{document}